\documentclass[12pt]{amsart}

\title {On some families of smooth affine spherical varieties of full rank}

\author{Kay Paulus}
\address{Department Mathematik, FAU Erlangen-N\"urnberg}
\email{paulus@math.fau.de}
  
\author{Guido Pezzini}
\address{Dipartimento di Matematica ``Guido Castelnuovo'', ``Sapienza'' Universit\`a di Roma}
\email{pezzini@mat.uniroma1.it}

\author{Bart Van Steirteghem}
\address{Department Mathematik, FAU Erlangen-N\"urnberg \& Department of Mathematics, Medgar Evers College - City University of New York}
\email{bartvs@mec.cuny.edu}

\usepackage{fullpage}
\usepackage{amssymb, amsmath}
\usepackage{mathrsfs}
\usepackage{enumerate}
\usepackage[colorlinks,breaklinks, allcolors=blue]{hyperref}
\usepackage{longtable}

\usepackage{pgf,tikz}
\usetikzlibrary{arrows}
\definecolor{sqsqsq}{rgb}{0.12549019607843137,0.12549019607843137,0.12549019607843137}
\usetikzlibrary[patterns]

\newtheorem{theorem}{Theorem}[section]
\newtheorem{lemma}[theorem]{Lemma}
\newtheorem{proposition}[theorem]{Proposition}
\newtheorem{corollary}[theorem]{Corollary}

\theoremstyle{definition}
\newtheorem{definition}[theorem]{Definition}
\newtheorem{remark}[theorem]{Remark}
\newtheorem{example}[theorem]{Example}
\newtheorem{List}[theorem]{List}

\numberwithin{equation}{section}

\newcommand{\supp}{\mathrm{supp}}
\newcommand{\C}{\mathbb C}
\newcommand{\Z}{\mathbb Z}
\newcommand{\N}{\mathbb N}
\newcommand{\Q}{\mathbb Q}
\newcommand{\R}{\mathbb R}

\newcommand{\A}{\mathbf A}

\newcommand{\GL}{\mathrm{GL}}
\newcommand{\SL}{\mathrm{SL}}
\newcommand{\Spin}{\mathrm{Spin}}
\newcommand{\SO}{\mathrm{SO}}
\newcommand{\Sp}{\mathrm{Sp}}
\newcommand{\SU}{\mathrm{SU}}
\newcommand{\Chi}{\mathcal X}

\newcommand{\fk}{\mathfrak{k}}

\newcommand{\ft}{\mathfrak{t}}

\newcommand{\su}{\mathfrak{su}}

\newcommand{\sA}{\mathsf{A}}
\newcommand{\sB}{\mathsf{B}}
\newcommand{\sC}{\mathsf{C}}
\newcommand{\sD}{\mathsf{D}}
\newcommand{\sE}{\mathsf{E}}
\newcommand{\sF}{\mathsf{F}}
\newcommand{\sG}{\mathsf{G}}

\DeclareMathOperator{\rk}{rk}

\DeclareMathOperator{\Hom}{Hom}
\DeclareMathOperator{\Lie}{Lie}

\DeclareMathOperator{\soc}{soc}

\newcommand{\<}{\langle}
\renewcommand{\>}{\rangle}

\newcommand{\onto}{\twoheadrightarrow}
\newcommand{\into}{\hookrightarrow}

\newcommand{\wm}{\Gamma}

\newcommand{\dw}{\Lambda^+}
\newcommand{\wl}{\Lambda}
\newcommand{\rl}{\Lambda_R}
\newcommand{\sr}{S}

\newcommand{\SNWM}{\Sigma^N(\wm)}
\newcommand{\slcst}{\SL(2)\times \C^{\times}}

\newcommand{\D}{\mathcal{D}}

\newcommand{\eps}{\varepsilon}
\newcommand{\om}{\omega}

\newcommand{\inn}{\subset}
\newcommand{\loccit}{{\em loc.cit.}}

\newcommand{\G}{\wm}

\renewcommand{\a}{\alpha}
\renewcommand{\d}{\delta}
\newcommand{\la}{\<}
\newcommand{\ra}{\>}

\renewcommand{\b}{\beta}
\usepackage{nicefrac}
\newcommand{\w}{\omega}
\newcommand{\bpm}{\begin{pmatrix}}
\newcommand{\epm}{\end{pmatrix}}

\renewcommand{\P}{\mathcal{P}}

\begin{document}

\begin{abstract}
Let $G$ be a complex connected reductive group. I.~Losev has shown that a smooth affine spherical $G$-variety $X$ is uniquely determined by its weight monoid, which is the set of irreducible representations of $G$ that occur in the coordinate ring of $X$. In this paper we use a combinatorial characterization of the weight monoids of smooth affine spherical varieties to classify: (a) all such varieties for $G=\SL(2) \times \C^{\times}$ and (b) all such varieties for $G$ simple which have a $G$-saturated weight monoid of full rank. We also use the characterization and F.~Knop's classification theorem for multiplicity free Hamiltonian manifolds to give a new proof of C.~Woodward's result that every reflective Delzant polytope is the moment polytope of such a manifold.
\end{abstract}

\maketitle

\section{Introduction} \label{sec:intro}
Spherical varieties play a role in many areas of mathematics, including symplectic geometry. This provides the main motivation for this paper, in which we combinatorially classify certain families of smooth affine spherical varieties, as described in more detail later in this introduction. We first recall, following \cite{knop-autoHam}, the connection between smooth affine spherical varieties and multiplicity free Hamiltonian manifolds. This will explain how explicit classifications of families of such varieties, like the ones in this paper, can be used to generate (families of) examples of multiplicity free Hamiltonian manifolds. We have endeavored to make the presentation accessible to non-experts, and we hope that it shows that the combinatorial theory of spherical varieties can be \emph{used} effectively. 

A \emph{Hamiltonian manifold} $M$ is a compact connected symplectic manifold equipped with an action of a compact connected Lie group $K$ and a $K$-equivariant moment map $\mu: M \to \fk^{*}$, where $\fk^*$ is the dual of the Lie algebra $\fk$ of $K$. Let us choose a maximal torus $T_{\R}$ in $K$ and a dominant Weyl chamber $\ft^+$ in the dual $\Lie(T_{\R})^*$ of the Lie algebra of $T_{\R}$. We can identify $\Lie(T_{\R})^*$ with the subspace of $\fk^*$ fixed by the coadjoint action of $T_{\R}$. In \cite{kirwan-convexity}, F. Kirwan showed that $\mu(M) \cap \ft^+$ is a convex polytope. We call it the \emph{moment polytope} of $M$ and denote it by $\mathcal{P}_M$. In \cite{guill&stern-mf}, V.~Guillemin and S.~Sternberg introduced an important class of such manifolds: a Hamiltonian manifold $M$ is called \emph{multiplicity free} if every symplectic reduction $\mu^{-1}(x)/K_x$ is a point. Here $x\in \mu(M)$ and $K_x = \{k \in K \colon k\cdot x = x\}$ is the isotropy group of $x$ under the coadjoint action of $K$. 

When $K$ is a (compact) torus, i.e.\ when $K = \mathrm{U}(1)^r$ for some $r \in \N$, and the action of $K$  is effective, then a multiplicity free Hamiltonian $K$-manifold is called a \emph{symplectic toric manifold}. In his influential 1988 paper \cite{delzant-abel}, T.~Delzant showed that symplectic toric manifolds are uniquely determined by their moment polytope. He also gave a combinatorial description of the polytopes that can occur as moment polytopes of symplectic toric manifolds: they are the so-called \emph{Delzant polytopes} (which are also known as simple regular polytopes). Building on work of M.~Brion \cite{brion-applicmoment}, R.~Sjamaar \cite{sjamaar-convexreex} and I.~Losev \cite{losev-knopconj}, F.~Knop generalized both of Delzant's results to nonabelian groups $K$ in \cite{knop-autoHam}, using smooth affine spherical varieties. In addition, Knop further extended these results recently to the setting of multiplicity free quasi-Hamiltonian manifolds in \cite{knop-qham-arxiv}, and, in \cite{woodward-spherical}, C.~Woodward used the theory of spherical varieties to study the existence of compatible K\"ahler structures on multiplicity free Hamiltonian manifolds.

Let us give more details on the results of \cite{knop-autoHam}. Theorem 10.2 of \loccit\ says that \emph{a multiplicity free Hamiltonian $K$-manifold $M$ is uniquely determined by its moment polytope $\mathcal{P}_M$ and its generic isotropy group $K_M$.} To state Knop's theorem describing the ``admissible'' moment polytopes we need some more notation. 

Let $T$ be the complexification of $T_{\R}$ and let $\wl:=\Hom(T,\C^{\times}) \cong \Z^{\dim T}$ be the character group of $T$. Then $T$ is a maximal torus of the complexification $G$ of $K$, which is a connected reductive group. The choice of the dominant Weyl chamber $\ft^+$ above corresponds to the choice of a Borel subgroup $B$ of $G$. Restriction of characters yields an identification of $\wl$ with the character group $\Hom(T_{\R},\C^{\times})$. Furthermore, $\lambda \mapsto (2\pi i)^{-1} d\lambda$ is  an embedding $\Hom(T_{\R},\C^{\times}) \into \Lie(T_{\R})^*$, where $d\lambda$ is the differential of $\lambda$ at the identity. We also identify $\wl$ with the image of this embedding, which is a sublattice of $\Lie(T_{\R})^*$ that spans $\Lie(T_{\R})^*$ as a vector space. When $a \in \ft^{+} \inn \fk^{*}$, we will use $G(a)$ for the complexification of the stabilizer $K_a:=\{k \in K \colon k\cdot a = a\}$.  It is well known that $T \inn G(a)$ and that $G(a)$ is a Levi subgroup of $G$. 
If $\mathcal{P}$ is a convex polytope in $\ft^+$ and $a \in \mathcal{P}$, then the \emph{tangent cone} to $\mathcal{P}$ at $a$ is 
\begin{equation}
C_a\mathcal{P}:= \R_{\ge 0}(\mathcal{P}-a).
\end{equation}
Observe that $C_a\mathcal{P}$ is a cone in
\begin{equation}
C_a\mathcal{\ft^+}:= \R_{\ge 0}(\ft^+-a) \inn \Lie(T_{\R})^*.
\end{equation}
Furthermore, $C_a\mathcal{\ft^+}$ is the dominant Weyl chamber for $G(a)$ corresponding to its Borel subgroup $B\cap G(a)$. 

As explained in \cite[pp.570-571]{knop-autoHam}, given $\mathcal{P}_M$ the information about the generic isotropy group $K_M$ of a multiplicity free Hamiltonian $K$-manifold $M$ can be encoded in a sublattice $\wl_M$ of $\wl$. Indeed, let $L_{\R} \inn K$ be the centralizer of $\mathcal{P}_M$ (seen as a subset of $\fk^*$), that is, $L_{\R}=\{k \in K \colon k\cdot p = p \text{ for every } p \in \mathcal{P}_M\}$. Then $L_{\R}$ is a Levi subgroup of $K$ containing $T_{\R}$ and $L_{\R}/K_M$ is a torus. Consequently, $K_M$ is determined by the character group 
\begin{equation}
\Lambda_M:=\Hom(L_{\R}/K_M, \C^{\times}),
\end{equation} 
which is a subgroup of $\Lambda \equiv \Hom(T_{\R},\C^{\times}).$

Knop showed that the moment polytopes of multiplicity free Hamiltonian manifolds locally ``look like'' the weight monoids of smooth affine spherical varieties (see Definition~\ref{def:weight_monoid} below). To be more precise,  Theorem 11.2 of \cite{knop-autoHam} tells us that \emph{a pair $(\mathcal{P}, \wl_0)$, with $\mathcal{P}$ a polytope in $\ft^+$ and $\wl_0$ a subgroup of $\wl$, is equal to $(\mathcal{P}_M,\wl_M)$  for some multiplicity free Hamiltonian $K$-manifold $M$ if and only if for every vertex $a$ of $\mathcal{P}$ there exists a smooth affine spherical $G(a)$-variety $X_a$ such that 
\begin{equation} \label{eq:sphericity}
\begin{split}
C_a\mathcal{P}&= \text{the cone generated by the weight monoid of $X_a$;  and}\\
\wl_0 &= \text{the abelian group generated by the weight monoid of $X_a$}.
\end{split}
\end{equation}
}
We recall that a normal $G$-variety is called \textbf{spherical} if it contains a dense orbit under the action of a Borel subgroup of $G$. 

Correspondingly, the varieties $X_a$ can be considered as ``building blocks'' for the manifold $M$. Let us recall briefly in which sense, referring to Section 2 of \cite{knop-autoHam} for details: the manifold $M$ admits an open cover $M=\bigcup_a U_a$, where the union is over all vertices $a$ of $\mathcal P_M$ and $U_a$ is $K$-stable for all $a$. Moreover, the subset $U_a$ is a $K$-equivariant fibre bundle of the form $K\times^{K_a} Y_a$, where $Y_a$ is a $K_a$-stable open subset of $X_a$ intersecting the closed $G(a)$-orbit.

\begin{example}\label{ex:su2}
Let $K=\SU(2)$. Then one can identify $\su(2)^*$ with $\R^3$ and the coadjoint action of $\SU(2)$ with the standard action through $\SU(2) \onto \SO(3,\R)$ on $\R^3$. Let $v\in \su(2)^*\setminus\{0\}$. Like every coadjoint orbit,  $S^2 \cong \SU(2)\cdot v$ is a Hamiltonian $\SU(2)$-manifold with moment map the inclusion $\iota:\SU(2)\cdot v \into \su(2)^*$. Consequently $M:=S^2 \times S^2$ is a Hamiltonian manifold under the diagonal action of $\SU(2)$ with moment map 
\[\mu\colon S^2 \times S^2 \to \su(2)^*, (x_1,x_2) \mapsto \iota(x_1) + \iota(x_2).\] One can check that $M$ is multiplicity free. 
For $T_{\R}$ we can take the subgroup of $K$ made up of diagonal matrices. Then $T_{R} \cong \mathrm{U}(1)$ and $\wl \cong \Z$. We can choose the identification $\su(2)^* \equiv \R^3$ in such a way that $\Lie(T_{\R})^*$ gets identified with $\R(0,0,1)$ and the weight lattice $\wl$ with $\Z(0,0,1)$. As a dominant Weyl chamber $\ft^+$ we can then choose $\R_{\ge 0}(0,0,1)$. To lighten notation, we will identify $\wl$ with $\Z$ and $\ft^+$ with $\R_{\ge 0}$. Elementary computations show that 
\[\mathcal{P}_M = \mu(M) \cap \ft^+ = [0,b],\]
where $b \in \R_{>0}$ depends on the choice of $v \in \su(2)^*$, that $L_{\R} = T_{\R}$ and that the generic isotropy group $K_M$ is equal to the center of $\SU(2)$. Consequently,
\[\wl_M = 2\Z \inn \Z \equiv \wl.\]  
We confirm that the conditions \eqref{eq:sphericity} are satisfied at the two vertices of $\mathcal{P}_M$. We begin with the vertex $0$. Since $K_0 = \SU(2)$, we have that $G(0)= \SL(2)$. Observe that $C_0 \mathcal{P} = \ft^+$. Consider the smooth affine spherical $\SL(2)$-variety
\[X_0 = \SL(2)/T\]
where $T$ is the subset of $\SL(2)$ made up of diagonal matrices. As is well known, the weight monoid of $X_0$ is $2\N \inn \N=\wl \cap \ft^+$, whence \eqref{eq:sphericity} holds at the vertex $0$. Next we consider the vertex $b$ of $\mathcal{P}_M$. Then $K_b = T_{\R}$ and consequently $G(b)=T$. Note that $C_b\mathcal{P}_M = -\ft^+$. Straightforward computations show that for the smooth affine spherical $T$-variety
\[X_b = \C \text{ with action $t\cdot v = t^{2}v$ for $t\in T, v\in X_b$}\] the conditions in \eqref{eq:sphericity} hold again. 
\end{example}

\begin{remark}
\begin{enumerate}[(a)]
\item Generalizing Example~\ref{ex:su2}, Corollary 11.4 in \cite{knop-autoHam} shows how one can recover  P.~Igl{\'e}sias's classification of multiplicity free Hamiltonian $\SU(2)$-manifolds in \cite{iglesias-so3} from Knop's theorems above.
\item Similarly, Delzant's classification of symplectic toric manifolds in \cite{delzant-abel} by Delzant polytopes is an immediate consequence of Knop's results (see \cite[Corollary 11.3]{knop-autoHam}.)  
\item In \cite[Section 11]{knop-qham-arxiv}, Knop has given examples of multiplicity free quasi-Hamiltonian manifolds $M$ and the corresponding pairs $(\mathcal{P}_M, \wl_M)$. In his thesis \cite{paulus_phd}, K.~Paulus has obtained classifications of certain subclasses of such manifolds using the results in \cite{knop-qham-arxiv}.  
\end{enumerate}
\end{remark}

For any given pair $(\mathcal{P},\wl_0)$ and vertex $a$ of $\mathcal{P}$, checking the conditions \eqref{eq:sphericity} can be reduced to a finite number of elementary verifications. Indeed, in \cite{charwm_arxivv2}  Pezzini and Van Steirteghem used the combinatorial theory of spherical varieties (always defined over $\C$ in this paper) and a smoothness criterion due to R.~Camus \cite{camus} to give a combinatorial characterization of the weight monoids of smooth affine spherical varieties.
 
We also recall that, in \cite{losev-knopconj}, Losev proved Knop's conjecture that a smooth affine spherical variety is uniquely determined by its weight monoid. Put differently, the classification of smooth affine spherical varieties is equivalent to the classification of their weight monoids. 
In this paper we apply the criterion in \cite{charwm_arxivv2} to classify the weight monoids of three families of smooth affine spherical varieties. 

More specifically, in Section \ref{sec:G_sat_full_rank} we find, for all simple groups $G$, the $G$-saturated smooth affine spherical varieties of full rank (see Definitions \ref{def:G-saturated} and \ref{def:fullrank} below). In Section \ref{sec:SL_2_times_C_star} we classify the smooth affine spherical $(\SL(2)\times \C^{\times})$--varieties. Finally, in Section~\ref{sec:woodward} we show that every so-called \emph{reflective monoid} which is ``Delzant'' is the weight monoid of a smooth affine spherical variety, see Theorem~\ref{thm:ourwoodward}. Thanks to Knop's aforementioned results on multiplicity free Hamiltonian manifolds, this yields a new proof of one of C.~Woodward's results in \cite{woodward-classif}: every reflective Delzant polytope is the moment polytope of a multiplicity free Hamiltonian manifold.

For completeness, we recall the existing classifications of smooth affine spherical varieties. In \cite{kramer}, M.~Kr\"amer classified affine spherical homogeneous spaces $G/H$ where $G$ is simple, while I.V.~ Mikityuk \cite{mikityuk} and Brion \cite{brion-sphhom} independently generalized this to arbitrary reductive $G$. Spherical modules (that is, representations of $G$ which are spherical when considered as $G$-varieties) were classified in \cite{kac-rmks, leahy, benson-ratcliff-mf}. Building on these two cases, Knop and Van Steirteghem classified all smooth affine
spherical varieties up to coverings, central tori, and $\C^{\times}$-fibrations in \cite{knop&bvs-classif}.
The weight monoids of the varieties in these classifications have been determined in \cite{kramer}, \cite{avdeev-extws}, \cite{leahy} (see also \cite[Section 5]{knop-rmks}) and \cite{knop&pezzini&bvs-wmsasv-prepar}, respectively. We remark that these are all classifications of so-called ``primitive'' smooth affine spherical varieties (see, e.g., \cite[Examples 2.3 -- 2.6]{knop&bvs-classif} for why the restriction to such varieties is necessary). This means that, while possible, obtaining the results of the present paper from these classifications would also require work. For example, the $(\SL(2)\times \C^{\times})$--varieties we classify in Section~\ref{sec:SL_2_times_C_star} (cf.~Table~\ref{table:sl2cst}) correspond to only four entries in the Tables of \cite{knop&bvs-classif}. 

\subsection*{Notation}
Throughout this paper, $G$ will be a complex connected reductive group, with a chosen Borel subgroup $B$ and maximal torus $T \inn B$. The weight lattice $\Hom(T,\C^{\times})$, identified with $\Hom(B,\C^{\times})$ as usual, will be denoted $\wl$. We will use $\rl$ for the root lattice and $\sr$ for the set of simple roots. When $\alpha \in \sr$, we will use $\alpha^{\vee}$ for the corresponding coroot, which we view as an element of $\Hom_{\Z}(\wl,\Z)$. We will use $\dw$ for the monoid of dominant weights. Recall that every weight $\lambda \in \dw$ naturally corresponds to an irreducible representation of $G$, which we will denote $V(\lambda)$. We will number the fundamental weights and the simple roots of the simple Lie algebras and of the corresponding simply connected simple algebraic group
as in \cite{bourbaki-geadl47}. All varieties are defined over $\C$ and are, by definition, irreducible. Unless stated otherwise, $X$ will denote an affine spherical $G$-variety.

When $F$ is a subset of $\wl \otimes_{\Z} \R$ and $\mathbb{A}$ is a submonoid of $(\R,+)$, like $\N, \Z, \Q$ or $\Q_{\ge 0}$, we will use the notation $\mathbb{A}F$ for the $\mathbb{A}$-span of $F$ in $\wl \otimes_{Z}\R$. For example, when $\wm \inn \wl$, then $\Z\wm$ and $\Q_{\ge 0}\wm$ are, respectively, the lattice and cone generated by $\wm$ in $\wl \otimes_{\Z} \Q$. If $F=\{f_1,f_2,\ldots,f_n\}$ is a finite set, then we will also write $\<f_1,f_2,\ldots,f_n\>_{\mathbb{A}}$ for $\mathbb{A}F$. If $\Chi$ is a lattice, then we will write $\Chi^{*}$ for the dual lattice $\Hom_{\Z}(\Chi, \Z)$. By a polytope (or
 a convex polytope) we mean the convex hull of finitely many points.

\subsection*{Acknowledgement}
The authors wish to thank Michel Brion and Baohua Fu for organizing the International Conference on Spherical Varieties at the Tsinghua Sanya International Mathematics Forum in November 2016. Van Steirteghem presented some of the results in this paper during his talk at the Conference, and is grateful for the invitation.

The authors thank Wolfgang Ruppert for several very helpful conversations about finitely generated abelian groups. They also thank the two referees for their careful reading and for useful remarks, suggestions and corrections on a previous version of this paper.

Much of the work on this project was done during visits by Van Steirteghem to the FAU in Erlangen in the summer of 2015 and the academic year 2016--2017. He would like to express his gratitude to Friedrich Knop and the Department of Mathematics at the FAU for hosting him. He received support from the City University of New York PSC-CUNY Research Award Program and from the National Science Foundation through grant DMS 1407394. He also thanks Medgar Evers College for his 2016-17 Fellowship Award. Pezzini was partially supported by the DFG Schwerpunktprogramm 1388 -- Darstellungstheorie. 

\section{Smooth affine spherical varieties} \label{sec:affine_spherical_varieties}
In this section we begin by briefly reviewing notions from the combinatorial theory of affine spherical varieties we will need. For further details and context, we refer to \cite{msfwm} and \cite{charwm_arxivv2}. In Theorem~\ref{thm:G_sat_char}, we recall from \cite{charwm_arxivv2} the characterization of $G$-saturated weight monoids of smooth affine spherical varieties. This will be the main tool we will use in Section~\ref{sec:G_sat_full_rank}. As throughout the paper, \emph{$X$ is an affine spherical $G$-variety.} As is well known, a normal affine $G$-variety $Y$ is spherical if and only if its coordinate ring $\C[Y]$ is multiplicity free as a $G$-module~\cite{vin&kim}.

\begin{definition} \label{def:weight_monoid}
Let $Y$ be an affine $G$-variety. The \textbf{weight monoid} of $Y$ is
\begin{equation}
\wm(Y) := \{\lambda \in \dw \colon \dim_{\C}\Hom^G(V(\lambda),\C[Y]) \neq 0\}.
\end{equation}
\end{definition}

Since $X$ is a normal variety, its weight monoid $\wm = \wm(X)$ is finitely generated and satisfies the following equality in $\wl \otimes_{\Z}\Q$:
\begin{equation}
\Z\wm \cap \Q_{\ge 0}\wm = \wm. \label{eq:normal_monoid}
\end{equation}
In this paper we assume all monoids to be finitely generated, and we call a submonoid $\wm$ of $\dw$ satisfying condition~\eqref{eq:normal_monoid} \textbf{normal}.

It is well-known that non-isomorphic affine spherical $G$-varieties may have the same weight monoid. For example, if $G=\SL(2)$ and $V=\<x^2, xy, y^2\>_{\C}$ is the vector space of binary forms of degree $2$ on which $G$ acts by linear change of variables, then the smooth affine spherical $G$-variety $G\cdot xy \inn V$ has the same weight monoid as the singular affine spherical $G$-variety $\overline{G\cdot x^2} \inn V$. In the 1990s, Knop conjectured that for \emph{smooth} affine spherical varieties the weight monoid is a complete invariant. This conjecture was proved by Losev: 
\begin{theorem}[{\cite{losev-knopconj}}] \label{thm:losev}
If $X_1$ and $X_2$ are smooth affine spherical $G$-varieties with $\wm(X_1) = \wm(X_2)$, then $X_1$ and $X_2$ are $G$-equivariantly isomorphic. 
\end{theorem}

This result leads to a natural question: which normal submonoids of $\dw$ are the weight monoids of smooth affine spherical varieties? 
\begin{definition} \label{def:smooth_weight_monoid}
Let $\wm$ be a submonoid of $\dw$. We will say that $\wm$ is a \textbf{smooth weight monoid} if there exists a smooth affine spherical $G$-variety $X$ such that $\wm(X) = \wm$.
\end{definition}

\begin{example}
Suppose $G=T$ is a torus. Then $\dw = \wl \cong \Z^{\dim T}$. A basic fact in the theory of toric varieties (see, e.g.\ \cite[Section 2.1]{fulton-toric}) says that a submonoid $\wm \inn \wl$ is a smooth weight monoid if and only if 
\[\wm = \<\lambda_1, -\lambda_1, \lambda_2, -\lambda_2, \ldots, \lambda_r, -\lambda_r, \lambda_{r+1}, \lambda_{r+2}, \ldots,\lambda_n\>_{\N}\]
for some collection $\lambda_1, \lambda_2, \ldots,\lambda_n$ of $\Z$-linearly independent elements of $\wl$ and some $r \in \{0,1,\ldots,n\}$. 
\end{example}

In contrast to the case where $G$ is a torus, the characterization of smooth weight monoids for general $G$ is much more complicated. In \cite[Theorem 4.2]{charwm_arxivv2}, Pezzini and Van Steirteghem gave such a combinatorial characterization. We recall a special case below in Theorem~\ref{thm:G_sat_char}. Before doing so and for completeness we quickly recall the combinatorial invariants $S^p(\wm)$, $\Sigma^N(\wm)$ and $S_{\wm}$ needed in the statement of the theorem. We omit most of the discussion of the geometric meaning of these invariants, but wish to emphasize that computing them only uses basic combinatorics of root systems and elementary calculations. 

Besides its weight monoid $\wm(Y)$, another important invariant of an affine  $G$-variety $Y$ is its set $\Sigma^N(Y)$ of N-spherical roots, which we now define. First, we recall that the \emph{root monoid} $\mathscr M_Y$ of $Y$ is the
submonoid of $\wl$ generated by 
\[\{\lambda+\mu-\nu \mid \lambda, \mu,
\nu \in \dw \text{ such that } \C[Y]_{(\nu)} \cap
(\C[Y]_{(\lambda)}\C[Y]_{(\mu)}) \neq 0\},\]
where for $\gamma \in \dw$ we used $\C[Y]_{(\gamma)}$ for the isotypic
component of type $\gamma$ in $\C[Y]$ and where
$\C[Y]_{(\lambda)}\C[Y]_{(\mu)}$ is the subspace of $\C[Y]$
spanned by the set $\{fg \colon f \in \C[Y]_{(\lambda)}, g \in
\C[Y]_{(\mu)}\}$.   By \cite[Theorem 1.3]{knop-auto}, the saturation of $\mathscr M_Y$, that is, the intersection of the cone spanned by $\mathscr M_Y$ and the group generated by $\mathscr M_Y$, is a free submonoid of $\rl$.

\begin{definition} \label{def:N_spherical_roots_variety}
Let $Y$ be an affine $G$-variety. The set $\Sigma^N(Y)$ of \textbf{N-spherical roots of $Y$} is the basis (as a monoid) of the saturation of the root monoid $\mathscr M_Y$ of $Y$. 
\end{definition}

\begin{definition} \label{def:N_adapted_spherical_roots}
Let $\wm$ be a submonoid of $\dw$ and let $\sigma \in \rl$. We say that $\sigma$ is \textbf{N-adapted} to $\wm$ if there exists an affine spherical $G$-variety $X$ such that $\Sigma^N(X) = \{\sigma\}$ and $\wm(X) = \wm$. We use $\Sigma^N(\wm)$ for the set of all $\sigma \in \rl$ that are N-adapted to $\wm$. 
\end{definition}

Corollary 2.17 of \cite{msfwm}, which we will recall in Proposition~\ref{prop:adapnsphroots_general}, gives a combinatorial description of $\Sigma^N(\wm)$, when $\wm$ is normal. In Proposition~\ref{prop:adapnsphroots} below we recall, from \cite{charwm_arxivv2}, the simpler description of $\Sigma^N(\wm)$ under the stronger assumption that $\wm$ be $G$-saturated.

\begin{remark}
In \cite[Theorem 1.2]{losev-knopconj}, Losev proved that an affine spherical $G$-variety $X$ (smooth or not) is uniquely determined (up to $G$-equivariant isomorphism) by the pair $\wm(X),\Sigma^N(X)$. In \cite{avdeev&cupit-newold-arxivv2}, R.~Avdeev and S.~Cupit-Foutou have proposed a proof that the root monoid of an affine spherical variety is free. 
\end{remark}

\begin{definition} \label{def:G-saturated}
Let  $\wm$ be a submonoid of the monoid $\dw$ of dominant weights of $G$.
We say that $\wm$ is \textbf{$G$-saturated} if the following
equality holds in $\wl$: 
\begin{equation}
\Z\wm \cap \dw = \wm. \label{eq:Gsat}
\end{equation}
We will also call an affine spherical $G$-variety $G$-saturated if its weight monoid is $G$-saturated. 
\end{definition}

\begin{remark}
Observe that by \eqref{eq:Gsat}, the lattice $\Z\wm$ determines $\wm$ when $\wm$ is $G$-saturated.
\end{remark}

The following characterization of $G$-saturated affine spherical varieties, due to Luna, will be useful later.

\begin{proposition}[{\cite[Proposition~2.31 and Remark~2.32]{charwm_arxivv2}}]\label{prop:G-saturatedness}
An affine spherical $G$-variety $X$ is $G$-saturated if and only if $\Sigma^N(X)$ does not contain any simple root and $X$ has no $G$-stable prime divisor.
\end{proposition}

\begin{definition} \label{def:scspherrootsG}
Let $\sigma$ be an element of the root lattice $\rl$ of $G$ and let $\sigma = \sum_{\alpha \in \sr} n_\alpha \alpha$ be its unique expression as a linear combination of the simple roots. The \textbf{support} of $\sigma$ is
\(\mathrm{supp}(\sigma)=\{\alpha \in \sr \colon n_\alpha \neq 0\}.\)
The \textbf{type} of $\mathrm{supp}(\sigma)$ is the Dynkin type of the root subsystem generated by $\mathrm{supp}(\sigma)$ in the root system of $G$. The \textbf{set $\Sigma^{sc}(G)$ of spherically closed spherical roots of $G$} is the subset of  $\N \sr$ defined as follows: an element $\sigma$ of $\N\sr$ belongs to $\Sigma^{sc}(G)$ if $\sigma$ is listed in Table~\ref{table:scspher}, where the numbering of the simple roots in $\mathrm{supp}(\sigma)$ is as in \cite{bourbaki-geadl47}.      
\end{definition}

\begin{table}\caption{spherically closed spherical roots} \label{table:scspher}
\begin{center}
\begin{tabular}{ll}
Type of $\supp(\sigma)$ & $\sigma$ \\
\hline
$\sf A_1$ & $\alpha_1$\\
$\sf A_1$ & $2\alpha_1$\\
$\mathsf A_1 \times \mathsf A_1$ & $\alpha_1+\alpha_1'$\\
$\mathsf A_n$, $n\geq 2$ & $\alpha_1+\ldots+\alpha_n$\\
$\mathsf A_3$ & $\alpha_1+2\alpha_2+\alpha_3$\\
$\mathsf B_n$, $n\geq 2$ & $\alpha_1+\ldots+\alpha_n$\\
                     & $2(\alpha_1+\ldots+\alpha_n)$\\
$\mathsf B_3$ & $\alpha_1+2\alpha_2+3\alpha_3$\\
$\mathsf C_n$, $n\geq 3$ & $\alpha_1+2(\alpha_2+\ldots+\alpha_{n-1})+\alpha_n$\\
$\mathsf D_n$, $n\geq 4$ & $2(\alpha_1+\ldots+\alpha_{n-2})+\alpha_{n-1}+\alpha_n$\\
$\mathsf F_4$ & $\alpha_1+2\alpha_2+3\alpha_3+2\alpha_4$\\
$\mathsf G_2$ & $4\alpha_1+2\alpha_2$\\          
& $\alpha_1+\alpha_2$
\end{tabular}
\end{center}
\end{table}

\begin{definition} \label{def:Sp}
Let $\wm$ be a set of dominant weights of $G$, that is $\wm \subset
\dw$. Then we define
\[S^p(\wm):= \{\alpha \in S \colon \<\alpha^{\vee},\lambda\> = 0
\text{ for all $\lambda \in \wm$} \}. \]
\end{definition}
 
\begin{proposition}[{\cite[Prop.~1.6]{charwm_arxivv2}}] \label{prop:adapnsphroots}
Suppose $\wm$ is a $G$-saturated submonoid of $\dw$. The set $\SNWM$ consists of all $\sigma \in \Sigma^{sc}(G)$ that satisfy all the following conditions:
\begin{enumerate}[(i)]
\item $\sigma$ is not a simple root; \label{item:osignotsimple}
\item $\sigma \in \Z\wm$; \label{item:osiginlat}
\item $\sigma$ is compatible with $S^p(\wm)$, that is: \label{compat}
\begin{itemize}
\item[-] if $\sigma=\alpha_1+\ldots+\alpha_n$ with support of type $\mathsf{B}_n$ then
$\{\alpha_2, \alpha_3, \ldots, \alpha_{n-1}\} \inn S^p(\wm)$ and $\alpha_n \notin S^p(\wm)$;
\item[-] if $\sigma=\alpha_1+2(\alpha_2+\ldots+\alpha_{n-1})+\alpha_n$ with support of type $\mathsf{C}_n$ then
$\{\alpha_3, \alpha_4, \ldots, \alpha_n\}
\inn S^p(\wm)$; 
\item[-] if $\sigma$ is any other element of $\Sigma^{sc}(G)$ then 
\(\{\alpha \in \supp(\sigma)\colon \<\alpha^{\vee}, \sigma\> =0\}
\inn S^p(\wm)\);
\end{itemize}
\item if $\sigma = 2 \alpha \in 2\sr$  
then $\<\alpha^{\vee}, \gamma\> \in 2\Z$ for all $\gamma \in \Z\wm$; \label{item:corooteven}
\item if $\sigma = \alpha+\beta$ with $\alpha, \beta \in S$ and $\alpha \perp \beta$, then $\<\alpha^{\vee}, \gamma\> = \<\beta^{\vee}, \gamma\>$ for all $\gamma \in \Z\wm$. \label{item:osigorthosum}
\end{enumerate} 
\end{proposition}

\begin{proposition}[{\cite[Prop.~1.7]{charwm_arxivv2}}] \label{prop:localizroots} Let $\wm$ be a $G$-saturated submonoid of $\dw$. 
Among all the subsets $F$ of $S$ such that the relative interior of the cone spanned by $\{\alpha^{\vee}|_{\Z\wm} \colon \alpha \in F\}$ in $\Hom_{\Z}(\Z\wm,\Q)$ intersects the cone $$\{\nu \in \Hom_{\Z}(\Z\wm,\Q) \colon \<\nu,\sigma\> \leq 0
\text{ for all } \sigma \in \Sigma^N(\wm)\}$$ 
there is a unique one, denoted $S_{\wm}$, that contains all the others.  
\end{proposition}

\begin{definition} \label{def:admissible}
Let $\sr$ be the set of simple roots of a root system. Let $\sr^p$ be a subset of $\sr$.  Let $\Sigma^N$ be a subset of $\N\sr$. We say that the triple $(\sr, \sr^p, \Sigma^N)$ is \textbf{admissible} if there exists a finite set $I$ and for every $i \in I$ a triple $(\sr_i, \sr^p_i, \Sigma_i)$ from List~\ref{list:pat} below and an automorphism $f_i$ of the Dynkin diagram of $\sr_i$ such that the Dynkin diagram of $S$ is the disjoint union over $i\in I$ of the Dynkin diagrams of the $\sr_i$, that $\sr^p = \cup_if_i(\sr^p_i)$ and that $\Sigma^N = \cup_i f_i(\Sigma_i)$. 
\end{definition}

\begin{List}[Primitive admissible triples] \label{list:pat}
\par
\begin{enumerate}[1.]
\item $(\sr,\sr,\emptyset)$ where $\sr$ is the set of simple roots
  of an irreducible root system; \label{item:admiss:triv}
\item $(\sA_n, \{\alpha_2,\alpha_3, \ldots,\alpha_{n}\}, \emptyset)$
  for $n \ge 1$; \label{item:admiss:An_standardmod}
\item $(\sA_n, \{\alpha_1, \alpha_3, \alpha_5, \ldots, \alpha_{n-1}\},\{\alpha_1+2\alpha_2+\alpha_3, \alpha_3+ 2\alpha_4 + \alpha_5, \ldots, \alpha_{n-3}+2\alpha_{n-2}+\alpha_{n-1}\})$ for $n \ge 4$, $n$ even; \label{item:admiss:An_even}
\item $(\sA_n \times \sA_k, \{\alpha_{k+2}, \alpha_{k+3}, \ldots, \alpha_n\}, \{\alpha_1+\alpha_1', \alpha_2+\alpha_2', \ldots, \alpha_k + \alpha_k'\})$ for $n > k \ge 2$;\label{item:admiss:AnAk}
\item $(\sC_n, \{\alpha_2,\alpha_3, \ldots,\alpha_{n}\}, \emptyset)$  for $n \ge 2$; \label{item:admiss:Cn}
\item $(\sD_5, \{\alpha_2, \alpha_3, \alpha_4\}, \{\alpha_2 + 2\alpha_3 + \alpha_4 + 2\alpha_5\})$. \label{item:admiss:D5}
\end{enumerate}
\end{List}
For example, $(\sA_2 \times \sD_5,\{\alpha_1, \alpha_2',\alpha_3',\alpha_5'\},\{\alpha_2'+2\alpha_3'+2\alpha_4'+\alpha_5'\})$ is an admissible triple.

\begin{remark} \label{rem:emptytriple}
\begin{enumerate}[(a)]
\item Note that we allow $I=\emptyset$ in
Definition~\ref{def:admissible}: the triple $(\emptyset, \emptyset, \emptyset)$ is
admissible.
\item For $n=1$, the triple (\ref{item:admiss:An_standardmod}.) in List~\ref{list:pat} is $(\sA_1, \emptyset,\emptyset)$.   
\end{enumerate}
\end{remark}

\begin{theorem}[{\cite[Theorem 1.12]{charwm_arxivv2}}] \label{thm:G_sat_char}
If $\wm$ is a $G$-saturated monoid of dominant weights of
$G$, then $\wm$ is the weight monoid of a smooth affine spherical $G$-variety if and only if 
\begin{enumerate}[(a)]
\item $\{\alpha^{\vee}|_{\Z\wm}\colon \alpha \in S_{\wm}\setminus S^p(\wm)\}$ is a
  subset of a basis of $(\Z\wm)^*$; and \label{cond:partofbasis}
\item for all $\alpha,\beta \in S_{\wm}\setminus S^p(\wm)$ such that $\alpha\neq\beta$ and $\alpha^{\vee}|_{\Z\wm} = \beta^{\vee}|_{\Z\wm}$ we have $\alpha+\beta \in \Z\wm$; and \label{cond:sumisNsphericalroot}
\item the triple $(S_{\wm},S^p(\wm), \Sigma^N(\wm) \cap \Z S_{\wm})$ is admissible. \label{cond:trip_admissible}
\end{enumerate}
\end{theorem}

\begin{remark} \label{rem:Gsatinvars}
As explained in \cite{charwm_arxivv2}, when $\wm$ is $G$-saturated, $S^p(\wm), \Sigma^N(\wm), S_{\wm}$ and $(S_{\wm},S^p(\wm), \Sigma^N(\wm) \cap \Z S_{\wm})$ are standard invariants from the theory of spherical varieties of the ``most generic'' affine spherical $G$-variety with weight monoid $\wm$ (see Section~2 of \loccit\ for details). In particular, if $\wm$ is a $G$-saturated weight monoid and $X$ is a smooth affine spherical $G$-variety with $\wm(X)=\wm$, then $\Sigma^N(\wm) = \Sigma^N(X)$. 
\end{remark}

The following well-known consequence of the \emph{Elementary Divisors Theorem} (see, e.g.\ \cite[Theorem 5.2, p.234]{lang_algebra_second_ed}) is useful in verifying condition~(\ref{cond:partofbasis}) of Theorem~\ref{thm:G_sat_char}.

\begin{lemma}\label{lem:part_of_basis}
Let $k,l \in \N$ with $k\geq l\geq 1$ and $A \in \operatorname{Mat}_{k\times l}(\Z)$. The following are equivalent:
\begin{enumerate}[(1)]
\item There exists a matrix $C \in \operatorname{Mat}_{k\times (k-l)}$ such that the $k\times k$ block matrix $(A|C)$ is invertible over $\Z$;
\item The greatest common divisor of all $l \times l$ minors in $A$ is equal to $1$. 
\end{enumerate}
\end{lemma}

We will use Theorem~\ref{thm:G_sat_char} in Section~\ref{sec:G_sat_full_rank} to determine the smooth weight monoids that are $G$-saturated and have full rank, where $G$ is a simple algebraic group. In his thesis \cite{wgkim-phd}, W.G.~Kim implemented Theorem~\ref{thm:G_sat_char} as an algorithm in \emph{Sage} \cite{sagemath2016} for the case where $G=\SL(n)$ and $\wm$ is free. He also used the Theorem to find the smooth weight monoids for $G=\SL(n)$ generated by fundamental weights, cf.\ \cite[Theorem 4.1]{wgkim-phd}.

\section{\texorpdfstring{$G$}{G}-saturated smooth affine spherical varieties of full rank} \label{sec:G_sat_full_rank}

In this Section we generalize \cite[Theorem 1.15]{charwm_arxivv2}, which states: \emph{if $G$ is a simply connected semisimple linear algebraic group, then there exists a smooth affine model $G$-variety if and only if the simple factors of $G$ are of type $\sA$ or of type $\sC$.} Recalling that an affine spherical $G$-variety is called a \emph{model variety} if its weight monoid is $\dw$, this theorem tells us for which semisimple and simply connected groups $\dw$ itself is a smooth weight monoid. Here we will determine \emph{all} the $G$-saturated weight monoids (and affine spherical varieties) of full rank that are smooth, where $G$ is simply connected and simple.   

\begin{definition} \label{def:fullrank}
We will say that a sublattice $\Chi$ of $\wl$ has \textbf{full rank} if $\rk \Chi = \rk \wl$. Furthermore, we say that a submonoid of $\dw$ has \emph{full rank} if the lattice it generates in $\wl$ has full rank. Finally, we say that an affine spherical $G$-variety has \emph{full rank} if its weight monoid has full rank. 
\end{definition}

\subsection{Restrictions on the set of N-adapted spherical roots} \label{subsec:restrictions}

We begin by collecting some immediate consequences of Theorem~\ref{thm:G_sat_char} that we will use in the sequel. As the following sets of spherical roots will frequently show up, we give them a name:
\begin{align*}
2\sr&:= \{2\alpha \colon \alpha \in \sr\}; \\
S^+&:= \{\alpha + \beta \colon \alpha, \beta \in \sr, \alpha \neq \beta, \alpha \not\perp \beta\}.
\end{align*}

\begin{lemma} \label{lem:SpSigmaN_GSat_FullRank}
Let $\wm$ be a submonoid of $\dw$.
\begin{enumerate}[(a)] 
\item If $\rk \Z\wm = \rk \wl$, then $S^p(\wm) = \emptyset$ \label{lemma_item:Sp_GSat_FullRank}
\item If $\rk \Z\wm = \rk \wl$ and $\wm$ is $G$-saturated, then $\Sigma^N(\wm) \inn 2S \cup S^+$. \label{lemma_item:SigmaN_GSat_FullRank}
\end{enumerate} 
\end{lemma} 
\begin{proof}
Assertion (\ref{lemma_item:Sp_GSat_FullRank}) follows from Definition~\ref{def:Sp} of $S^p(\wm)$. Indeed, if $\alpha \in S^p(\wm)$, then $\Z\wm$ is a subset of the lattice $\{\lambda \in \wl \colon \<\alpha^{\vee}, \lambda\> = 0\}$, so that $\rk \Z\wm < \rk \wl$. Part (\ref{lemma_item:SigmaN_GSat_FullRank}) is a consequence of Proposition~\ref{prop:adapnsphroots}. Indeed, it follows from the fact that $S^p(\wm) = \emptyset$ and condition (\ref{compat}) in Proposition~\ref{prop:adapnsphroots} that if $\sigma \in \Sigma^N(\wm)$, then $\sigma \in 2S \cup S^+$ or $\sigma = \alpha + \alpha'$ with $\alpha,\alpha' \in \sr$ and $\alpha \perp \alpha'$. The second possibility cannot occur by condition (\ref{item:osigorthosum}) of  Proposition~\ref{prop:adapnsphroots} since $\rk \Z\wm = \rk \wl$. 
\end{proof}

\begin{lemma} \label{lem:alpha_not_in_supp_Sgamma}
Let $\wm$ be a $G$-saturated submonoid of $\dw$. 
If $\alpha \in \sr \setminus \supp(\Sigma^N(\wm))$, then $\alpha \in S_{\wm}$. 
\end{lemma}
\begin{proof}
Observe that if $\alpha \notin \supp(\Sigma^N(\wm))$, then $\<\alpha^{\vee}, \sigma\> \leq 0$ for all $\sigma \in \Sigma^N(\wm)$. 
\end{proof}

For the remainder of Section~\ref{subsec:restrictions}, we will \emph{assume that $\wm$ is a smooth, $G$-saturated submonoid of $\dw$ of full rank.}

\begin{lemma} \label{lem:Gsatsmooth_Sgamma}
 \begin{enumerate}[(a)]
\item If $\alpha, \beta \in S_{\wm}$ with $\alpha \neq \beta$, then $\alpha \perp \beta$. \label{lemma_item_Sgamma_A1}
\item If $\alpha \in S_{\wm}$, then $2\alpha \notin \Sigma^N(\wm)$. \label{lemma_item_Sgamma_double}
\end{enumerate}
\end{lemma}
\begin{proof}
Part (\ref{lemma_item_Sgamma_A1}) follows from Theorem~\ref{thm:G_sat_char}(\ref{cond:trip_admissible}) and both assertions of Lemma~\ref{lem:SpSigmaN_GSat_FullRank}: the only primitive admissible triple in List~\ref{list:pat} of which the second component is equal to $\emptyset$ and the third component is a subset of $2S \cup S^+$ is $(\sA_1, \emptyset, \emptyset)$.  Part (\ref{lemma_item_Sgamma_double}) also immediately follows from Theorem~\ref{thm:G_sat_char}(\ref{cond:trip_admissible}): there is no triple in List~\ref{list:pat} for which the third component contains an element of $2\sr$. 
\end{proof}

\begin{lemma} \label{lem:Gsat_two_simple_roots}
Let $\alpha, \beta \in \sr$ with $\alpha \neq \beta$ and $\alpha \not\perp \beta$. Then the following hold:
\begin{enumerate}[(a)]
\item $\alpha \in \supp(\Sigma^N(\wm))$ and $\beta \in \supp(\Sigma^N(\wm))$. \label{lemma_item_1:Gsat_two_simple_roots} 
\item Suppose $\<\alpha^{\vee}, \beta\> \in \{-1,-3\}$. If $2\alpha \in \SNWM$, then $\alpha+\beta \notin \Sigma^N(\wm)$. \label{lemma_item_3:Gsat_two_simple_roots}  
\end{enumerate}
\end{lemma}
\begin{proof}
We first show that $\alpha \in \supp(\Sigma^N(\wm))$ or $\beta \in \supp(\Sigma^N(\wm))$. Indeed, if not then we would have $\alpha,\beta \in S_{\wm}$ by Lemma~\ref{lem:alpha_not_in_supp_Sgamma}, which contradicts Lemma~\ref{lem:Gsatsmooth_Sgamma}(\ref{lemma_item_Sgamma_A1}). To prove (\ref{lemma_item_1:Gsat_two_simple_roots}), we can now assume that $\alpha \in \supp(\SNWM)$. If $\beta \notin \supp(\SNWM)$ held, then, using that $\alpha \not\perp \beta$, we would have that $\<\beta^{\vee}, \sigma\> <0$ for all $\sigma \in \SNWM$ with $\alpha \in \supp(\sigma)$. By Proposition~\ref{prop:localizroots} it would follow that $\{\alpha,\beta\} \inn S_{\wm}$, contradicting Lemma~\ref{lem:Gsatsmooth_Sgamma}(\ref{lemma_item_Sgamma_A1}). This proves (\ref{lemma_item_1:Gsat_two_simple_roots}). Assertion (\ref{lemma_item_3:Gsat_two_simple_roots}) follows from conditions (\ref{item:osiginlat}) and (\ref{item:corooteven}) in Proposition~\ref{prop:adapnsphroots}, since $\<\alpha^{\vee}, \alpha + \beta\> \in \{-1,1\}$. 
\end{proof}

\begin{lemma} \label{lem:Gsat_three_simple_roots}
Let $\alpha,\beta,\gamma \in \sr$ with $\alpha\perp\gamma$, $\<\alpha^{\vee}, \beta\>=-1$ and $\gamma \not \perp \beta$. Then the following hold:
\begin{enumerate}[(a)]
\item If $2\alpha \in \SNWM$, then $\beta + \gamma \notin \SNWM$. \label{lemma_item_1:Gsat_three_simple_roots}
\item Suppose $\<\beta^{\vee},\alpha\> = -1$. If $\alpha + \beta \in \SNWM$, then $\beta + \gamma \in \SNWM$. \label{lemma_item_2:Gsat_three_simple_roots}
\end{enumerate}
\end{lemma}

\begin{proof}
Assertion (\ref{lemma_item_1:Gsat_three_simple_roots}) follows from conditions (\ref{item:osiginlat}) and (\ref{item:corooteven}) in Proposition~\ref{prop:adapnsphroots}, since $\<\alpha^{\vee}, \beta + \gamma\> =-1$. 

We now prove (\ref{lemma_item_2:Gsat_three_simple_roots}). By Lemma~\ref{lem:Gsat_two_simple_roots}(\ref{lemma_item_1:Gsat_two_simple_roots}) we know that $\gamma \in \supp(\SNWM)$ and by Lemma~\ref{lem:Gsat_two_simple_roots}(\ref{lemma_item_3:Gsat_two_simple_roots}) that $2\beta \notin \SNWM$. To obtain a contradiction, we assume that $\beta + \gamma \notin \SNWM$. Then, by Lemma~\ref{lem:SpSigmaN_GSat_FullRank}(\ref{lemma_item:SigmaN_GSat_FullRank}), we know that $2\gamma \in \SNWM$, or $\gamma + \delta \in \Sigma^N(\wm)$ for some $\delta \in S\setminus\{\gamma\}$ with $\gamma \not\perp \delta$. We prove that both are impossible. 

We first show that $2\gamma \in \SNWM$ cannot happen. If $\<\gamma^{\vee}, \beta\> = -1$, then this follows from part (\ref{lemma_item_1:Gsat_three_simple_roots}). Next we consider the case $\<\gamma^{\vee}, \beta\> \neq -1$. Then $\<\gamma^{\vee}, \beta\> = -2$ and the root subsystem generated by $\{\alpha, \beta, \gamma\}$ is of type $\sB_3$. If $2 \gamma \in \SNWM$, then by Lemma~\ref{lem:SpSigmaN_GSat_FullRank}(\ref{lemma_item:SigmaN_GSat_FullRank}) and Lemma~\ref{lem:Gsat_two_simple_roots}(\ref{lemma_item_3:Gsat_two_simple_roots}), $2\gamma$ and $\alpha+\beta$ are the only elements of $\SNWM$ with $\beta$ or $\gamma$ in their support. Since $\<2\beta^{\vee}+\gamma^{\vee}, \alpha+\beta\>=0$ and $\<2\beta^{\vee}+\gamma^{\vee}, 2\gamma\>=0$
it follows by Proposition~\ref{prop:localizroots} that $\{\beta,\gamma\} \inn S_{\wm}$. This contradicts Lemma~\ref{lem:Gsatsmooth_Sgamma}(\ref{lemma_item_Sgamma_A1}).
 
We have shown that $2\gamma \notin \SNWM$, and consequently, $\gamma + \delta \in \Sigma^N(\wm)$ for some $\delta \in S\setminus\{\gamma\}$ with $\gamma \not\perp \delta$ (there may be two such $\delta \in S$ if the component of the Dynkin diagram of $S$ containing $\alpha, \beta$ and  $\gamma$ is of type $\sD_n, \sE_6, \sE_7$ or $\sE_8$ and $\gamma$ is a node of degree $3$). As is well known, the Dynkin diagram of $S$ contains no cycles, and consequently, $\alpha \perp \delta$ and $\delta \perp \beta$. Thus, $\<\beta^{\vee} + \gamma^{\vee}, \alpha+\beta\> \leq 0$ and $\<\beta^{\vee} + \gamma^{\vee}, \gamma + \delta\> \leq 0$. If the component of the Dynkin diagram of $S$ containing $\alpha, \beta$ and  $\gamma$ is of type $\sD_n, \sE_6, \sE_7$ or $\sE_8$ and $\beta$ is a node of degree $3$, then there may exist $\alpha' \in S\setminus\{\alpha,\beta,\gamma\}$ such that $\alpha'+\beta \in \SNWM$. In that case, $\<\beta^{\vee} + \gamma^{\vee}, \alpha'+\beta\> = 0$. Since for every other $\sigma \in \SNWM$ we have $\beta, \gamma \notin \supp(\sigma)$, and therefore $\<\beta^{\vee}, \sigma\> \leq 0$ and $\<\gamma^{\vee}, \sigma\> \leq 0$, it follows by Proposition~\ref{prop:localizroots} that $\{\beta,\gamma\} \inn S_{\wm}$, again contradicting Lemma~\ref{lem:Gsatsmooth_Sgamma}(\ref{lemma_item_Sgamma_A1}). 
\end{proof}

\subsection{Simple groups of type \texorpdfstring{$\sA$}{A}}
Applying Theorem~\ref{thm:G_sat_char} to $G$-saturated monoids of full rank for $G$ simple and simply connected of type $\sA$ we obtain the following proposition.
\begin{proposition} \label{prop:SL}
Let $G = \SL(n+1)$ for $n\ge 1$. If $\wm$ is a $G$-saturated submonoid of $\dw$ of full rank, then $\wm$ is a smooth weight monoid if and only if one of the following holds
\begin{enumerate}[(1)]
\item $2\rl \inn \Z\wm \inn 2\wl$; \label{propitem_SL_1} 
\item $n$ is even and $S^+ \inn \Z\wm$;  \label{propitem_SL_2}
\item $n$ is odd and $S^+ \inn \Z\wm$ and $\{\alpha_1^{\vee}|_{\Z\wm}, \alpha_3^{\vee}|_{\Z\wm}, \ldots \alpha_n^{\vee}|_{\Z\wm}\}$ is part of a basis of $(\Z\wm)^*$. \label{propitem_SL_3}
\end{enumerate}
\end{proposition}

\begin{remark}
For $n=1$, Proposition~\ref{prop:SL} asserts that the smooth weight monoids of rank $1$ for $G=\SL(2)$ are $\N(k\omega)$ with $k\in \{1,2,4\}$, where $\omega$ is the fundamental dominant weight (every normal submonoid of $\dw$ is $G$-saturated for $G=\SL(2)$).  
\end{remark}

\begin{proof} 
It is well-known (and an exercise using Theorem~\ref{thm:G_sat_char} to show) that for $\SL(2)$ the only smooth weight monoids of rank $1$ are $\N(k\omega)$ with $k\in \{1,2,4\}$. What is left is to prove the Proposition for $n\ge 2$. 

We first prove the ``only if'' statement in the proposition. It follows from Lemma~\ref{lem:SpSigmaN_GSat_FullRank}(\ref{lemma_item:SigmaN_GSat_FullRank}), Lemma~\ref{lem:Gsat_two_simple_roots} and Lemma~\ref{lem:Gsat_three_simple_roots} that $\SNWM=2S$ or $\SNWM = S^+$. If $\SNWM = 2S$ then $2\rl \inn \Z\wm$ by Proposition~\ref{prop:adapnsphroots}(\ref{item:osiginlat}), while $\Z\wm \inn 2\wl$ by Proposition~\ref{prop:adapnsphroots}(\ref{item:corooteven}). On the other hand, if $\SNWM = S^+$ and $n$ is odd, then a straightforward computation shows that $S_{\wm} = \{\alpha_1,\alpha_3,\ldots,\alpha_n\}$. Condition (\ref{cond:partofbasis}) in Theorem~\ref{thm:G_sat_char} then finishes the proof of the ``only if'' statement.

We now prove the reverse implication. If (\ref{propitem_SL_1}) holds, then one deduces from Proposition~\ref{prop:adapnsphroots} that $\SNWM = 2S$. A straightforward computation then shows that $S_{\wm} = \emptyset$, and therefore $(S_{\wm}, S^p(\wm), \SNWM \cap \Z S_{\wm}) = (\emptyset, \emptyset, \emptyset)$, which is admissible. Consequently, $\wm$ is smooth by Theorem~\ref{thm:G_sat_char}. If we have (\ref{propitem_SL_2}) or (\ref{propitem_SL_3}), then one computes using Proposition~\ref{prop:adapnsphroots} that $\SNWM = S^+$. If $n$ is even, then one computes that $S_{\wm}=\emptyset$, and it follows again that $\wm$ is smooth.  If $n$ is odd, then $S_{\wm} = \{\alpha_1,\alpha_3,\ldots,\alpha_n\}$. Consequently, $(S_{\wm}, S^p(\wm), \SNWM \cap \Z S_{\wm}) = (\{\alpha_1,\alpha_3,\ldots,\alpha_n\}, \emptyset, \emptyset)$, which is admissible. Condition~(\ref{cond:partofbasis}) of Theorem~\ref{thm:G_sat_char} is met by $\wm$ because it is part of the assumption (\ref{propitem_SL_3}) in the Proposition, and condition~(\ref{cond:sumisNsphericalroot}) of the Theorem is trivially met. We have shown that, once again, $\wm$ is smooth. 
\end{proof}

\begin{remark}
 The ``if'' statement in Proposition~\ref{prop:SL} could also be proved by exhibiting for every $\wm$ satisfying (\ref{propitem_SL_1}), (\ref{propitem_SL_2}) or (\ref{propitem_SL_3}) a smooth affine spherical $\SL(n+1)$-variety with $\wm(X) = \wm$. We list those varieties in Table~\ref{table:SL}.
\end{remark}

In the next lemma, we describe the lattices $\Z\wm$ in Proposition~\ref{prop:SL} more explicitly.
\begin{lemma} \label{lem:SL_lattices}
Let $\Chi$ be a sublattice of  the weight lattice $\wl$ of $\SL(n+1)$, where $n\ge 1$.
\begin{enumerate}[(a)]
\item We have $\rl \inn \Chi \inn \wl$ if and only if $\Chi = \<\alpha_2, \alpha_3, \ldots, \alpha_n, d\omega_n\>_{\Z}$ for some $d\in \N$ with $d|(n+1)$. \label{lemma_item_SL_lattices_0}
\item Suppose $n$ is even.  The lattice $\Chi$ has full rank and contains $S^+$ if and only if $\Chi = \Z S^+ \oplus \Z(k\omega_{n-1})$ for some $k \in \N\setminus\{0\}$. \label{lemma_item_SL_lattices_1}
\item Suppose $n$ is odd. The lattice $\Chi$ has full rank, contains $S^+$ and $\{\alpha_1^{\vee}|_{\Chi},\alpha_3^{\vee}|_{\Chi}, \ldots, \alpha_n^{\vee}|_{\Chi}\}$ is part of a basis of $\Chi^*$   if and only if  
$$\Chi = \<\alpha_2+\alpha_3, \alpha_3+\alpha_4, \ldots,\alpha_{n-1} + \alpha_{n}, e\omega_{n-1}, r\omega_{n-1} + \omega_n\>_{\Z}$$
for some $e,r\in \N$ with $e|\frac{n+1}{2}$ and $0\leq r \leq e-1$.  \label{lemma_item_SL_lattices_2}
\end{enumerate}
\end{lemma}
\begin{proof}
We begin with the proof of assertion (\ref{lemma_item_SL_lattices_0}). The direct computation
\(\textstyle{\sum_{k=1}^n}k\alpha_k = (n+1)\omega_n\)
implies that $\rl = \<\alpha_2,\alpha_3, \ldots, \alpha_n, (n+1)\omega_n\>_{\Z}$.  One readily checks that $\wl = \<\alpha_2,\alpha_3, \ldots, \alpha_n,\omega_n\>_{\Z}$. The assertion follows.

For the rest of the proof we put $\sigma_i = \alpha_i + \alpha_{i+1}$ for every $i \in \{1,2,\ldots,n-1\}$. Sublattices $\Chi$ of $\wl$ of full rank and containing $S^+$ are in one-to-one correspondence 
with subgroups of $\wl/\Z S^+$ of maximal rank.
We now prove assertion (\ref{lemma_item_SL_lattices_1}) and assume that $n\ge 2$ is even. One checks that
\[
\biggl({\textstyle \sum_{k=1}^{\frac{n}{2}-1}} k (\sigma_{2k-1}+ \sigma_{2k})\biggr) + \tfrac{n}{2}\sigma_{n-1} = \tfrac{n}{2} \omega_{n-1} + \omega_n.
\]
Clearly, this implies that 
\[\Z S^+ = \<\sigma_2, \sigma_3,\ldots, \sigma_{n-1}, \tfrac{n}{2} \omega_{n-1} + \omega_n\>_{\Z}.\]
Next, consider the matrix $A \in \operatorname{Mat}_{n\times n}(\Z)$ whose $(i,j)$-th entry is
\[A_{ij} = 
\begin{cases}
\<\alpha_j^{\vee},\sigma_{i+1}\> &\text{if $1\leq i\leq n-2$};  \\
\<\alpha_j^{\vee},\tfrac{n}{2} \omega_{n-1} + \omega_n\> &\text{if $i=n-1$};\\
\< \alpha_j^{\vee},\omega_{n-1}\> &\text{if $i=n$}.
\end{cases}
\]
It is straightforward to compute that $\det(A) = 1$. This implies that 
\[\wl = \Z S^+ \oplus \Z \omega_{n-1}.\] Assertion (\ref{lemma_item_SL_lattices_1}) follows. 

We now move to the proof of the assertion (\ref{lemma_item_SL_lattices_2}). Assume that $n$ is odd, and put $d=\frac{n+1}{2}$. One checks that
\[
{\textstyle \sum_{k=1}^{d-1}} k (\sigma_{2k-1}+ \sigma_{2k})= d \omega_{n-1}
\]
Clearly, this implies that 
\[\Z S^+ = \<\sigma_2, \sigma_3,\ldots, \sigma_{n-1}, d \omega_{n-1}\>_{\Z}.\]
Next, consider the matrix $C \in \operatorname{Mat}_{(n-1)\times n}(\Z)$ whose $(i,j)$-th entry is
\[C_{ij} = 
\begin{cases}
\< \alpha_j^{\vee},\sigma_{i+1}\> &\text{if $1\leq i\leq n-2$};  \\
\<\alpha_j^{\vee},d \omega_{n-1}\> &\text{if $i=n-1$}
\end{cases}
\]
For example, for $n=5$, we have
\[C = \begin{pmatrix}
-1 & 1 & 1 & -1 & 0 \\
0 & -1 & 1 & 1 & -1 \\
0 & 0 & -1 & 1 & 1 \\
0 & 0 & 0 & 3 & 0 \\
\end{pmatrix}
\]
It follows by a standard argument that $\wl/\Z S^+ \cong \Z \oplus \Z/d\Z$ and that the inverse image of the torsion $\Z/d\Z$ under the quotient map $\wl \onto \wl/\Z S^+$ is
\[\<\sigma_2, \sigma_3,\ldots, \sigma_{n-1}, \omega_{n-1}\>_{\Z}.\]
One deduces that if $\Chi$ is a lattice of rank $n$ containing $\Z S^+$, then 
$$\Chi = \<\sigma_2, \sigma_3, \ldots,\sigma_{n-1}, e\omega_{n-1}, r\omega_{n-1} + m\omega_n\>_{\Z}$$
for some $e,r,m \in \N$ with $e|d$, $0\leq r \leq e-1$ and $m\neq 0$.
Applying Lemma~\ref{lem:part_of_basis} one shows that for such a lattice $\Chi$, the set $\{\alpha_1^{\vee}|_{\Chi},\alpha_3^{\vee}|_{\Chi}, \ldots, \alpha_n^{\vee}|_{\Chi}\}$ is part of a basis of $\Chi^*$ if and only if $m=1$. This proves the Lemma.
\end{proof}

\begin{remark}
We observe that it follows from Lemma~\ref{lem:SL_lattices} that there are infinitely many lattices $\Z\wm$ (and correspondingly infinitely many non-isomorphic smooth affine spherical $\SL(n)$-varieties) satisfying (\ref{propitem_SL_2}) in Proposition~\ref{prop:SL}. On the other hand, there are only finitely many lattices  $\Z\wm$ satisfying (\ref{propitem_SL_3}) in Proposition~\ref{prop:SL}: in fact there are $\sigma(\frac{n+1}{2})$ such lattices, where $\sigma(\frac{n+1}{2})$ is the sum of all divisors of $\frac{n+1}{2}$. We also observe that there are $\tau(n+1)$ lattices $\Z\wm$ satisfying (\ref{propitem_SL_1}) of Proposition~\ref{prop:SL}, where $\tau(n+1)$ is the number of divisors of $n+1$.  
\end{remark}

\begin{corollary} \label{cor:SL}
Let $G=\SL(n+1)$ with $n\ge 1$. A $G$-saturated submonoid $\wm$ of $\dw$ of full rank is a smooth weight monoid if and only if 
\begin{enumerate}[(1)]
\item $\Z\wm= 2\<\alpha_2, \alpha_3, \ldots, \alpha_n, d\omega_n\>_{\Z}$ for some $d\in \N$ with $d|(n+1)$;
\item $n$ is even and $\Z\wm = \<\alpha_1+\alpha_2, \alpha_2+\alpha_3, \ldots, \alpha_{n-1}+\alpha_n, k\omega_{n-1}\>_{\Z}$ for some $k\in \N \setminus \{0\}$; or
\item $n$ is odd and $\Z\wm=\<\alpha_2+\alpha_3, \alpha_3+\alpha_4, \ldots,\alpha_{n-1} + \alpha_{n}, e\omega_{n-1}, r\omega_{n-1} + \omega_n\>_{\Z}$
for some $e,r\in \N$ with $e|\frac{n+1}{2}$ and $0\leq r \leq e-1$.
\end{enumerate}
\end{corollary}

\begin{remark}
In Table~\ref{table:SL} and in the proof of Corollary~\ref{cor:SL_varieties} below, the special orthogonal group $\SO(n+1)$ is defined as the set of matrices $g\in \SL(n+1)$ such that $g\cdot {}^tg=\textrm{Id}$, and when $n$ is odd, we choose $\Sp(n+1)$ inside $\SL(n+1)$ to be the subgroup fixing the skew-symmetric bilinear form given by the matrix
\[
J=\left(
\begin{array}{ccccc}
0 & 1 & \ldots & 0 & 0 \\
-1 & 0 & \ldots & 0 & 0 \\
\vdots & \vdots & \ddots & \vdots & \vdots \\
0 & 0 & \ldots & 0 & 1 \\
0 & 0 & \ldots & -1 & 0
\end{array}
\right)
\]
When $n$ is even, as in case~5 of Table~\ref{table:SL}, we embed $\Sp(n)$ into $\SL(n+1)$ via $A\mapsto \left(\begin{smallmatrix}A & 0 \\ 0 & 1 \end{smallmatrix}\right)$.
\end{remark}

\begin{remark}
In the proofs of Corollaries~\ref{cor:SL_varieties} and~\ref{cor:other_varieties} below, we use references where the sets of N-spherical roots of certain varieties are computed. To be precise, in these references one finds the sets of their {\em spherical roots}, which can be defined for an affine spherical variety $Z$ as the primitive elements in $\Z\Gamma(Z)$ that are positive rational multiples of the elements of $\Sigma^N(Z)$. Denote by $\Sigma(Z)$ the set of spherical roots of $Z$. The relation between $\Sigma^N(Z)$ and $\Sigma(Z)$ was described by Losev in \cite[Theorem 2]{losev-uniqueness} and for more details on these sets we refer to \cite{vs-interpretations}. Here we will only need the following simple observation: in general, the elements of $\Sigma^N(Z)$ appear in Table~\ref{table:scspher}, and are positive integer multiples of the elements of $\Sigma(Z)$. This is enough to conclude that $\Sigma^N(Z)=\Sigma(Z)$ in all cases we refer to during the aforementioned proofs, as one can check case-by-case using the description of $\Sigma(Z)$ given in the references used in the proofs of the Corollaries.
\end{remark}

\begin{corollary} \label{cor:SL_varieties}
Let $G=\SL(n+1)$ with $n\ge 1$. Every $G$-saturated smooth affine spherical $G$-variety of full rank is $G$-equivariantly isomorphic to a (unique) $G$-variety in Table~\ref{table:SL}.
\end{corollary}
\begin{longtable}{|c|c|c|c|c|c|}
\caption{$G$-saturated smooth affine spherical $\SL(n+1)$-varieties of full rank.}
\label{table:SL}\\
\hline
Case & $n$ & $X$ & $\Z\wm(X)$ & $\Sigma^N(X)$ & parameters \\
\hline
1 & $1$ &$\C^2$ & $\<\omega_1\>_{\Z}$& $\emptyset$ &  \\
\hline 
2 & $1$ &$\SL(2)/T$ & $\<2\omega_1\>_{\Z}$& $\{2\alpha_1\}$ &  \\
\hline 
3 & $1$ &$\SL(2)/N(T)$ & $\<4\omega_1\>_{\Z}$& $\{2\alpha_1\}$ &  \\
\hline 
4 & $\ge 2$ &\begin{tabular}{c}$\SL(n+1)/Z_d\SO(n+1)$\\
{\scriptsize $Z_d := \{\zeta\mathrm{Id}\colon \zeta \in \C, \zeta^{d}=1\}$}\\
{\scriptsize if $n$ is even,}\\
{\scriptsize $Z_d := \langle \text{diag}(\xi^\frac{n+1}d,\ldots,\xi^\frac{n+1}d,(-\xi)^\frac{n+1}d)\rangle$} \\
{\scriptsize $\xi$=primitive $2(n+1)$-th root of 1}\\
{\scriptsize if $n$ is odd}\\
\end{tabular}
& \begin{tabular}{c} $2\<\alpha_2,\alpha_3,\ldots,\alpha_n,$ \\ $d\omega_n\>_{\Z}$ \end{tabular} & $2S$ & $d|(n+1)$ \\
\hline 
5 & \begin{tabular}{c}$\ge 2$,\\ even\end{tabular} &
\begin{tabular}{c}$\SL(n+1)/H_k\Sp(n)$\\
{\scriptsize $H_k:= \{\textrm{diag}(\zeta,\ldots,\zeta,\zeta^{-n})\colon$} \\ {\scriptsize $\zeta \in \C, \zeta^{2k}=1\}$} \end{tabular}
 & $\Z S^+ \oplus \Z(k\omega_{n-1})$ & $S^+$ & $k \in \N\setminus\{0\}$\\
\hline 
6 & \begin{tabular}{c}$\ge 2$,\\odd\end{tabular} &  \begin{tabular}{c} $\SL(n+1)\times^{Z_{2e}\Sp(n+1)} \C^{n+1}$ \\ {\scriptsize $Z_{2e}$ defined as in case 4,} \\ {\scriptsize $\zeta\mathrm{Id} \in Z_{2e}$ acts on $\C^{n+1}$}\\ {\scriptsize as the multiplication by $\zeta^{r}$} \end{tabular} &
\begin{tabular}{c}$\<\alpha_2+\alpha_3, \alpha_3+\alpha_4,$\\ $\ldots,\alpha_{n-1}+\alpha_n,$\\
$ e\omega_{n-1},$ \\ $r\omega_{n-1} + \omega_n \>_{\Z}$ \end{tabular}& $S^+$ & \begin{tabular}{c}
$e|\frac{n+1}{2}$;\\ $0\leq r$ \\ $\leq e-1$ 
\end{tabular} \\
\hline 
\end{longtable}

\begin{proof}[Proof of Corollary~\ref{cor:SL_varieties}]
We first show that the varieties in Table~\ref{table:SL} do have the lattices specified in the table. Then we will prove that each variety of the table is $G$-saturated. By the ``only if'' part of Proposition~\ref{prop:SL} and Losev's Theorem~\ref{thm:losev}, this will finish the proof.

For the three $\SL(2)$-varieties the claims above are well known, so we only deal with the remaining varieties.

To avoid confusion, during the proof we denote by $X$ the smooth affine variety having $G$-saturated weight monoid, and such that $\Z\Gamma(X)$ is the one given in the table, and we denote by $Y$ instead of $X$ the ``candidate'' variety given in the second column of the table. By Remark~\ref{rem:Gsatinvars} and the proof of Proposition~\ref{prop:SL}, the set $\Sigma^N(X)$ is the one given in the fifth column of the table.

Let us consider the first remaining case, namely case~4, for the group $G=\SL(n+1)$ where $n\geq 2$. The lattice $\Z\Gamma(X)$ is $2\<\alpha_2,\alpha_3,\ldots,\alpha_n,d\omega_n\>_{\Z}$ where $d|(n+1)$. Notice that $\Z\Gamma(X)$ is contained in the lattice $2\<\alpha_2,\alpha_3,\ldots,\alpha_n,\omega_n\>_{\Z}$, which is $2\wl$, and the corresponding quotient has order $d$. Moreover $\Z\Gamma(X)$ contains $2\<\alpha_2,\alpha_3,\ldots,\alpha_n,(n+1)\omega_n\>_{\Z}$, which is $2\rl$.

The variety $Y$ given in the table is $G$-homogeneous, and the stabilizer indicated in the table (call it $H$) contains $\SO(n+1)$ and is contained in $N_G(\SO(n+1))$. More precisely, the quotient $N_G(\SO(n+1))/\SO(n+1)$ is cyclic of order $n+1$: it is generated by the class of the matrix $\zeta_0\textrm{Id}$ (where $\zeta_0$ is a primitive $(n+1)$-th root of unity) if $n+1$ is odd, and by the class of the matrix $\text{diag}(\xi,\ldots,\xi,-\xi)$ (where $\xi$ is a primitive $2(n+1)$-th root of unity) if $n+1$ is even. The quotient $H/\SO(n+1)$ is the subgroup of $N_G(\SO(n+1))/\SO(n+1)$ of order $d$.

Set $Y_0=G/\SO(n+1)$ and $Y_1=G/N_G(\SO(n+1))$. Recall that if $Z$ is an affine spherical variety then $\Z\Gamma(Z)$ is the lattice of $B$-eigenvalues of $B$-eigenvectors of $\C(Z)$. Then $\Z\Gamma(Y_0)\supset \Z\Gamma(Y)\supset \Z\Gamma(Y_1)$, and, by \cite[Lemma~2.4]{gandini-spherorbclos}, the quotient $\Z\Gamma(Y_0)/\Z\Gamma(Y)$ is isomorphic to the character group of $H/\SO(n+1)$, so it is cyclic of order $d$.

On the other hand it is well-known that $\Z\Gamma(Y_0)=2\Lambda$ (see e.g.\ \cite[Tabelle~1]{kramer}), and $\Z\Gamma(Y_1)=2\Lambda_R$ (because $\Z\Gamma(Y_1)=\Z\Sigma^N(Y_1)$ and $\Sigma^N(Y_1)=2S$, see e.g.\ \cite[Section~5.2]{luna-typeA}). This implies that $\Z\Gamma(Y_0)\supset \Z\Gamma(X)\supset \Z\Gamma(Y_1)$.

The desired equality $\Z\Gamma(X)=\Z\Gamma(Y)$ follows, because both quotients $\Z\Gamma(Y_0)/\Z\Gamma(Y)$ and $\Z\Gamma(Y_0)/\Z\Gamma(X)$ are quotients of the cyclic group $\Z\Gamma(Y_0)/\Z\Gamma(Y_1)$ and have the same order.

We discuss case~5, with the lattice $\Z S^+ \oplus \Z(k\omega_{n-1})$ where $k \in \N\setminus{0}$, for $G=\SL(n+1)$ where $n$ is even. The proof goes similarly as the previous case: we set $Y_0=G/\Sp(n)$ and $Y_1=G/N_G(\Sp(n))$, and we observe that the fibers of the natural map $Y_0\to Y$ have cardinality $k$, hence this is also the index of $\Z\Gamma(Y)$ in $\Z\Gamma(Y_0)$.

We have that $\Z S^+=\Z\Gamma(Y_1)$, because $\Z\Gamma(Y_1)=\Z\Sigma^N(Y_1)$ and $\Sigma^N(Y_1)=S^+$ by \cite[Section~3,~Case~31]{bravi-pezzini-reductive}, therefore $\Z S^+\subset \Z\Gamma(Y)$. Now $\Z\Gamma(Y_0)$ has a unique subgroup that contains $S^+$ and is of index $k$, namely $\Z\Gamma(X)$. The desired equality $\Z\Gamma(X)=\Z\Gamma(Y)$ follows.

Let us consider case~6. The lattice $\Z\Gamma(X)$ is $\<\alpha_2+\alpha_3, \alpha_3+\alpha_4,\ldots,\alpha_{n-1}+\alpha_n, e\omega_{n-1}, r\omega_{n-1} + \omega_n \>_{\Z}$ where $e|\frac{n+1}{2}$ and $0\leq r \leq e-1$, for $G=\SL(n+1)$ where $n$ is odd. Here again $\Sigma^N(X)=S^+$.

The variety $Y$ is here $\SL(n+1)\times^{Z_{2e}\Sp(n+1)}\C^{n+1}$, where $\textrm{diag}(\zeta)\in Z_{2e}$ acts on $\C^{n+1}$ as the multiplication by $\zeta^{r}$. Set $Y_0=\SL(n+1)\times^{\Sp(n+1)} \C^{n+1}$. It is well known that $\Z\Gamma(Y_0)$ is the weight lattice of $G$ and that $\Sigma^N(Y_0)=S^+$, see e.g.\ \cite[Section~3.3]{luna-model}. The map
\[
\begin{array}{cccc}
\pi\colon &Y_0=\SL(n+1)\times^{\Sp(n+1)} \C^{n+1} & \to & Y=\SL(n+1)\times^{Z_{2e}\Sp(n+1)} \C^{n+1} \\
&  [g,v] & \mapsto & [g,v]
\end{array}
\]
is surjective and $G$-equivariant, and its generic fibers have cardinality equal to the index of $\Z\Gamma(X)$ in $\Z\Gamma(Y_0)$.

Let us define $B$-eigenvectors of $\C(Y_0)$ of $B$-eigenvalues $\omega_n$ and $\omega_{n-1}$. Denote by $H$ the stabilizer in $\SL(n+1)$ of the point $[e_G,e_{n+1}]\in Y$, where $e_G$ is the neutral element of $G$ and $(e_i)_{i\in\{1,\ldots,n+1\}}$ is the standard basis of $\C^{n+1}$. Then the open $G$-orbit of $Y_0$ is isomorphic to the homogeneous space $G/H$. The map $gH\mapsto g_{n+1,n+1}$, where $g=(g_{i,j})_{i,j\in\{1,\ldots,n+1\}}\in G$, is a $B$-eigenvector of $\C(G/H)$, hence of $\C(Y_0)$, of $B$-eigenvalue $\omega_{n}$. On the other hand, a $B$-eigenvector of $\C[Y_0]$ of $B$-eigenvalue $\omega_{n-1}$ is the map
\[
[g,v]\mapsto \textrm{Pf}\left(
\begin{array}{cc}
a_{n,n} & a_{n,n+1} \\
a_{n+1,n} & a_{n+1,n+1} \\
\end{array}
\right)
\]
where $a=g J ({}^tg)$. Then it is elementary to check that the functions having weights resp.\ $e\omega_{n-1}$ and $r\omega_{n-1}+\omega_n$ descend to $B$-semiinvariant rational functions on $Y$. Moreover, since $\pi\colon Y_0\to Y$ is finite, we have $\Sigma^N(Y)=\Sigma^N(Y_0)$ up to replacing some elements with positive rational multiples. But Table~\ref{table:scspher} and the fact that $\Sigma^N(Y_0)=S^+$ imply that $\Sigma^N(Y)=S^+$. We conclude that $\alpha_i+\alpha_{i+1}$ is in $\Z\Gamma(Y)$ for all $i\in \{2,\ldots,n-1\}$.

It follows that $\Z\Gamma(Y)\supset\Z\Gamma(X)$, and since they are two sublattices of same index in $\Z\Gamma(Y_0)$, they are equal.

To finish the proof, it remains to prove that $Y$ is $G$-saturated, where $Y$ is any of the last three cases of the table. We use Proposition~\ref{prop:G-saturatedness}, and we must check that $\Sigma^N(Y)$ does not contain any simple root, and that $Y$ has no $G$-stable prime divisor.

For cases~4 and~5, recall that we have defined a homogeneous variety $Y_1$ such that $\Sigma^N(Y_1)$ doesn't contain any simple root. Then $Y_1$ is $G$-saturated by \loccit, and $Y$ too, because there exists a surjective $G$-equivariant finite map $Y\to Y_1$ as we have seen. We also deduce that $X=Y$, whence $\Sigma^N(Y)$ is indeed the set indicated in the fifth column of the table.

Finally, we consider case~6. Here we have defined a variety $Y_0$ such that $\Sigma^N(Y_0)=S^+$. The variety $Y_0$ has no $G$-stable prime divisor: indeed, a $G$-stable prime divisor of $Y_0$ would intersect the subset $V=\{ [e_G,v]\colon v\in\C^{n+1}\}$ of $Y_0$ in a $\Sp(n+1)$-stable prime divisor of $V\cong \C^{n+1}$, which is absurd.

Then $Y_0$ is $G$-saturated, and so is $Y$ because of the surjective $G$-equivariant finite map $\pi\colon Y_0\to Y$ that we have defined. We have already shown that $\Sigma^N(Y)=S^+$, so the proof is complete.
\end{proof}

\subsection{Other types}

We now classify the smooth and $G$-saturated affine spherical varieties of full rank, when $G$ is a simply connected simple group not of type $\sA$.

\begin{proposition} \label{prop:MonoidsOtherTypes}
Suppose $G$ is a simply connected simple group of type different from $\sA$. A $G$-saturated submonoid $\wm$ of $\dw$ of full rank is a smooth weight monoid if and only if one of the following holds:
\begin{enumerate}
\item $2\rl \inn \Z\wm \inn 2\wl$;
\item $G$ is of type $\sB_n$ with $n \ge 2$ and $\<\alpha_1+\alpha_2, \alpha_2+\alpha_3, \ldots, \alpha_{n-1}+\alpha_n, 2\alpha_n\>_{\Z} \inn \Z\wm \inn \<\om_1, \om_2, \ldots,\om_{n-1},2\om_n\>_{\Z}$; or
\item $G$ is of type $\sC_n$ with $n\ge 2$ and $\Z\wm = \wl$.
\end{enumerate}
\end{proposition} 
The proof of this proposition will be given on page \pageref{proof:prop:MonoidsOtherTypes}
\begin{remark}
The lattices $\Z\wm$ satisfying $2\rl \inn \Z\wm \inn 2\wl$ are in natural bijective correspondence with the subgroups of the quotient $2\wl/2\rl \cong \wl/\rl$. 
For all simply connected simple groups $G$, the quotient $\wl/\rl$ can be found in \cite[Planches I-IX]{bourbaki-geadl47}: 
\begin{longtable}{|c|c|}
\hline
Type of $G$ & $\wl/\rl$ \\
\hline
$\sB_n, n\ge 2$ & $\Z/2\Z$ \\
\hline
$\sC_n, n\ge 3$ & $\Z/2\Z$ \\
\hline
$\sD_n, n\ge 4$ even & $\Z/2\Z \times \Z/2\Z$ \\
\hline
$\sD_n, n\ge 5$ odd & $\Z/4\Z$ \\
\hline
$\sE_6$ & $\Z/3\Z$ \\
\hline
$\sE_7$ & $\Z/2\Z$ \\
\hline
$\sE_8$ & $\{0\}$ \\
\hline
$\sF_4$ & $\{0\}$ \\
\hline
$\sG_2$ & $\{0\}$ \\
\hline
\end{longtable}
\end{remark}

\begin{remark}
In Table~\ref{table:othertypes} and in the proof of Corollary~\ref{cor:other_varieties} below, the orthogonal group is defined as the stabilizer of the symmetric bilinear form given by the matrix $J$ that has entries equal to $1$ on the skew diagonal, and zeros elsewhere. We denote by $(e_i)_{i\in\{1,\ldots,2n\}}$ the standard basis of $\C^{2n}$ (with $n\geq 4$), and if $n$ is even then we define the subgroup $\SO(n)\times \SO(n)$ of $\SO(2n)$ as the stabilizer of the subspace of $\C^{2n}$ generated by $e_{\frac n2+1},e_{\frac n2+2},\ldots,e_{\frac{3n}2}$. If $n$ is odd, we define $\SO(n)\times \SO(n)$ as the stabilizer of any subspace of $\C^{2n}$ where the above bilinear form is nondegenerate. For the exceptional groups appearing in homogeneous spaces written e.g.\ as $\frac{\sE_6}{\sC_4}$, the numerator refers to the simply connected simple group, and the Dynkin types appearing in the denominator refer to connected semisimple subgroups of the corresponding type.
\end{remark}

\begin{corollary}\label{cor:other_varieties}
Let $G$ be a simply connected simple group of type different from $\sA$. Every $G$-saturated smooth affine spherical $G$-variety of full rank is $G$-equivariantly isomorphic to a (unique) $G$-variety in Table~\ref{table:othertypes}.
\end{corollary}

\begin{longtable}{|c|c|c|c|c|}
\caption{$G$-saturated smooth affine spherical $G$-varieties of full rank with $G$ simple and not of type $\sA$.}
\label{table:othertypes}\\
\hline
Case & $G$ & $X$ & $\Z\wm(X)$ & $\Sigma^N(X)$ \\
\hline
1 & $\Spin(2n+1), n\ge 2$ & \raisebox{-8pt}{\rule{0pt}{22pt}}$\frac{\SO(2n+1)}{\mathrm{S}(\mathrm{O}(n+1)\times\mathrm{O}(n))}$  & $2\rl$ &  $2S$ \\
\hline 
2 & $\Spin(2n+1), n\ge 2$ & \raisebox{-8pt}{\rule{0pt}{22pt}}$\frac{\SO(2n+1)}{\SO(n+1)\times\SO(n)}$  & $2\wl$ & $2S$ \\
\hline
3 & $\Spin(2n+1), n\ge 2$ & \raisebox{-8pt}{\rule{0pt}{22pt}}$\frac{\SO(2n+1)}{N(\GL(n))}$  & $\Z(S^+ \cup \{2\alpha_n\})$ &  $S^+ \cup \{2\alpha_n\}$ \\
\hline 
4 & $\Spin(2n+1), n\ge 2$ & \raisebox{-8pt}{\rule{0pt}{22pt}}$\frac{\SO(2n+1)}{\GL(n)}$  & $\<\om_1, \om_2, \ldots,\om_{n-1},2\om_n\>_{\Z}$ & $S^+ \cup \{2\alpha_n\}$ \\
\hline
5 & $\Sp(2n), n\ge 3$ & \raisebox{-8pt}{\rule{0pt}{22pt}}$\frac{\Sp(2n)}{N(\GL(n))}$ & $2\rl$ & $2S$ \\
\hline
6 & $\Sp(2n), n\ge 3$ & \raisebox{-8pt}{\rule{0pt}{22pt}}$\frac{\Sp(2n)}{\GL(n)}$ & $2\wl$ & $2S$ \\
\hline 
7 & $\Sp(2n), n\ge 2$ &
\begin{tabular}{c}
$\Sp(2n) \times ^{\Sp(2a)\times\Sp(2b)} \C^{2b}$\\
{\scriptsize $a=b=n/2$} \\
{\scriptsize if $n$ is even,}\\
{\scriptsize $a=b-1=(n-1)/2$}\\
{\scriptsize if $n$ is odd}
\end{tabular} & $\wl$ & $S^+$ \\
\hline
8 & $\Spin(2n),n \ge 4$ & \raisebox{-8pt}{\rule{0pt}{22pt}}$\frac{\SO(2n)}{N(\SO(n)\times\SO(n))}$ & $2\rl$ & $2S$ \\
\hline
9 & $\Spin(2n),n \ge 4$ & \raisebox{-8pt}{\rule{0pt}{22pt}}$\frac{\SO(2n)}{\SO(n)\times\SO(n)}$ & $2\wl$ & $2S$ \\
\hline
10 & $\Spin(2n),n \ge 4$ & \raisebox{-8pt}{\rule{0pt}{22pt}}$\frac{\SO(2n)}{\mathrm{S}(\mathrm{O}(n)\times\mathrm{O}(n))}$ & \begin{tabular}{c}$2\<\alpha_1,\ldots,\alpha_{n-2},\alpha_{n},$ \\ $\omega_1\>_\Z$ \end{tabular} & $2S$ \\
\hline
11 & $\Spin(2n),n \ge 4$ even & \begin{tabular}{c}\raisebox{-8pt}{\rule{0pt}{22pt}}$\frac{\SO(2n)}{\<A\cdot(\SO(n)\times\SO(n))\>}$, \\ {\scriptsize where $Ae_i=-e_{n-i+1}$,} \\ {\scriptsize $Ae_{\frac n2+i}=e_{\frac n2-i+1}$,} \\ {\scriptsize $Ae_{n+i}=e_{2n-i+1}$,} \\ {\scriptsize $Ae_{\frac{3n}2 +i}=-e_{\frac{3n}2 -i+1}$} \\ {\scriptsize  for all $i\in\{1,\ldots,n/2\}$} \end{tabular}& \begin{tabular}{c}$2\<\alpha_1,\ldots,\alpha_{n-2},\alpha_{n},$ \\ $\omega_n\>_\Z$ \end{tabular} & $2S$ \\
\hline
12 & $\Spin(2n),n \ge 4$ even & \raisebox{-8pt}{\rule{0pt}{22pt}}$\frac{\SO(2n)}{\<A\cdot\mathrm{S}(\mathrm{O}(n)\times\mathrm{O}(n))\>}$ & \begin{tabular}{c} $2\<\alpha_1,\ldots,\alpha_{n-2},\alpha_{n},$ \\ $\omega_1+\omega_n\>_\Z$ \end{tabular} & $2S$ \\
\hline
13 & $\sE_6$ & \raisebox{-8pt}{\rule{0pt}{22pt}}$\frac{\sE_6}{Z(\sE_6)\sC_4}$ & $2\rl$ & $2S$ \\
\hline
14 & $\sE_6$ & \raisebox{-8pt}{\rule{0pt}{22pt}}$\frac{\sE_6}{\sC_4}$ & $2\wl$ & $2S$ \\ 
\hline
15 & $\sE_7$ & \raisebox{-8pt}{\rule{0pt}{22pt}}$\frac{\sE_7}{N(\sA_7)}$ & $2\rl$ & $2S$ \\
\hline
16 & $\sE_7$ & \raisebox{-8pt}{\rule{0pt}{22pt}}$\frac{\sE_7}{\sA_7}$ & $2\wl$ & $2S$ \\
\hline
17 & $\sE_8$ & \raisebox{-8pt}{\rule{0pt}{22pt}}$\frac{\sE_8}{\sD_8}$ & $2\wl = 2\rl$ & $2S$ \\
\hline
18 & $\sF_4$ & \raisebox{-8pt}{\rule{0pt}{22pt}}$\frac{\sF_4}{\sC_3 \times \sA_1}$ & $2\wl = 2\rl$ & $2S$ \\
\hline
19 & $\sG_2$ & \raisebox{-8pt}{\rule{0pt}{22pt}}$\frac{\sG_2}{\sA_1 \times \sA_1}$ & $2\wl = 2\rl$ & $2S$ \\
\hline
\end{longtable}

\begin{proof}[Proof of Corollary~\ref{cor:other_varieties}]
We proceed as in the proof of Corollary~\ref{cor:SL_varieties}. Again, during the proof we denote by $X$ the smooth affine variety having $G$-saturated weight monoid, and such that $\Z\Gamma(X)$ is the one given in the table, and we denote by $Y$ instead of $X$ the ``candidate'' variety given in the third column of the table. By the proof of Proposition~\ref{prop:MonoidsOtherTypes}, the set $\Sigma^N(X)$ is the one given in the last column of the table. For the cases where the variety $Y$ is homogeneous, let us denote by $G_0$ and $H_0$ the groups and subgroups specified in the third column of the table such that $Y=G_0/H_0$; notice that $G_0$ is an isogenous quotient of $G$.

Let us consider the entries of the table where the variety $Y$ is homogeneous, and the corresponding stabilizer $H_0$ is equal to the normalizer $N_{G_0}(H^\circ_0)$ of its connected component $H^\circ_0$ containing the neutral element. They are cases 1, 3, 5, 8, 13, 15, 17, 18, 19.

These varieties are given in the paper \cite{bravi-pezzini-reductive} (they are, respectively, cases 9 with $p=q$, 33, 13, 15 with $p=q$, 21, 25, 27, 29, 30 of \loccit) together with the set $\Sigma^N(Y)$ which is equal to $\Sigma^N(X)$. Moreover, for such varieties we have that $\Z\Gamma(Y)$ is equal to the lattice generated by $\Sigma^N(Y)$ (see \loccit), and this is precisely $\Z\Gamma(X)$. Notice that for these cases $Y$ is $G$-saturated thanks to Proposition~\ref{prop:G-saturatedness}, since it is homogeneous and $\Sigma^N(Y)$ does not contain simple roots.

Let us consider cases 2, 4, 6, 9, 10 with $n$ odd, 14, 16. They are all homogeneous, and for each of them the variety $Y_1=G_0/N_{G_0}(H^\circ_0)$ appears as one of the cases considered before. We have seen that $Y_1$ is $G$-saturated, whence $Y$ has this property too. In all cases one checks that $\Z\Gamma(X)$ contains $\Z\Gamma(Y_1)$, and that the index of $H_0$ in $N_{G_0}(H^\circ_0)$ is equal to the index of $\Z\Gamma(Y_1)$ in $\Z\Gamma(X)$. We conclude that $\Z\Gamma(Y_1)$ has the same index in $\Z\Gamma(X)$ and in $\Z\Gamma(Y)$ by \cite[Lemma~2.4]{gandini-spherorbclos}. Using this fact, one checks using Proposition~\ref{prop:MonoidsOtherTypes} together with the fact that $Y$ is smooth and $G$-saturated that the only possibility for $\Z\Gamma(Y)$ is $\Z\Gamma(X)$. We underline that this argument is valid for case 10 with $n$ odd because the quotient $2\Lambda/2\Lambda_R$ is cyclic, and the lattice of case 10 (with $n$ odd) corresponds to the unique subgroup of order $2$ of this quotient.

We turn to the cases 10 with $n$ even, 11, and 12. They are similar to case 10 with $n$ odd, but here $2\Lambda/2\Lambda_R\cong (\Z/2\Z)\times (\Z/2\Z)$, and the three cases under consideration correspond to the three subgroups of order $2$ of this quotient.

Since we treat these varieties together, let us rename then as $Y_1=\SO(2n)/H_1$, $Y_2=\SO(2n)/H_2$, and $Y_3=\SO(2n)/H_3$, where $H_1= \mathrm{S}(\mathrm{O}(n)\times\mathrm{O}(n))$, $H_2 = \<A\cdot(\SO(n)\times\SO(n))\>$, and $H_3=\<A\cdot\mathrm{S}(\mathrm{O}(n)\times\mathrm{O}(n))\>$.

Notice that $N_G(\SO(n)\times\SO(n))/(\SO(n)\times\SO(n))$ is isomorphic to $(\Z/2\Z)\times (\Z/2\Z)$, and is generated by the classes of $A$ and of an element in $\mathrm{S}(\mathrm{O}(n)\times\mathrm{O}(n))\setminus (\SO(n)\times\SO(n))$. By \cite[Lemma~2.4]{gandini-spherorbclos}, the lattice $\Z\Gamma(Y_i)$ contains $\Z\Gamma(\SO(2n)/N_G(\SO(n)\times\SO(n)))$ as a sublattice of index $2$, and is a sublattice of index $2$ of $\Z\Gamma(\SO(2n)/(\SO(n)\times\SO(n)))$. To check that the lattice $\Z\Gamma(Y_i)$ is the one indicated in the table, it is enough to exhibit a rational function on $\SO(2n)$ that is $H_i$-invariant by right multiplication, and $B$-semiinvariant of weight $2\omega_1$ and $2\omega_n$ for resp.\ $i=1$ and $i=2$. 

A function for $i=1$ with the desired properties is the one that associates to $g\in \SO(2n)$ the value $v(g)\cdot J \cdot {}^tv(g)$, where $v(g)$ is obtained from the last row of $g$ setting to $0$ the first $n/2$ and the last $n/2$ entries. A function for $i=2$ is the one that associates to $g\in\SO(2n)$ the $n\times n$-minor given by the last $n$ rows and the middle $n$ columns of the matrix $g$.

Finally, it remains to discuss the only case in the table where $Y$ is not homogeneous, namely case 7. It is well known that $\Z\Gamma(Y)=\Lambda$ and that $\Sigma^N(Y)=S^+$, see e.g.\ \cite[Section~3.3]{luna-model}. The variety $Y$ has no $G$-stable prime divisors, by the same argument used for the variety $\SL(n+1)\times^{\Sp(n+1)}\C^{n+1}$ in the proof of Corollary~\ref{cor:SL_varieties}. Then $Y$ is $G$-saturated, again by Proposition~\ref{prop:G-saturatedness}, and the proof is complete.
\end{proof}

\begin{lemma} \label{lem:Sigma_N_other_types_full_rank}
Let $G$ be a simple group of type different from $\sA$ and let $\wm \inn \dw$ be a $G$-saturated submonoid of full rank. If $\wm$ is smooth then one of the following holds:
\begin{enumerate}[(a)]
\item $\SNWM = 2S$;
\item $G$ is of type $\sC_n$ with $n\ge 2$ and $\SNWM = S^+$; or
\item $G$ is of type $\sB_n$ with $n\ge 2$ and $\SNWM = S^+ \cup \{2\alpha_n\}$.
\end{enumerate}
\end{lemma}
\begin{proof}
By Lemma~\ref{lem:Gsat_two_simple_roots}(\ref{lemma_item_1:Gsat_two_simple_roots}) and Lemma~\ref{lem:SpSigmaN_GSat_FullRank}(\ref{lemma_item:SigmaN_GSat_FullRank}) we know that 
$\supp(\SNWM) = S$ and $\SNWM \inn S^+ \cup 2S$.
We first suppose that the Dynkin diagram of $G$ is simply laced (i.e. $G$ is of type $\sD_n$, $\sE_6$, $\sE_7$ or $\sE_8$). It follows from repeated applications of Lemma~\ref{lem:Gsat_two_simple_roots}(\ref{lemma_item_3:Gsat_two_simple_roots}) and Lemma~\ref{lem:Gsat_three_simple_roots}(\ref{lemma_item_1:Gsat_three_simple_roots}) that if $2\alpha_1 \in \SNWM$, then $\SNWM = 2S$. 
On the other hand, if $2\alpha_1 \notin \SNWM$, then $\alpha_1+\alpha_2 \in \SNWM$ and it then follows from the same lemmas that $\SNWM = S^+$. 
We claim that Lemma~\ref{lem:Gsatsmooth_Sgamma}(\ref{lemma_item_Sgamma_A1}) implies that $\SNWM = S^+$ is not possible. We show this by contradiction: we suppose that $\SNWM = S^+$ and show that $S_{\wm}$ contains two non-orthogonal simple roots. Indeed,
\begin{itemize}
\item[-] if $G$ is of type $\sD_n$ with $n\ge 4$, then one checks that $\<\alpha_{n-3}^{\vee}+\alpha_{n-2}^{\vee}+\alpha_{n-1}^{\vee} + \alpha_n^{\vee},\sigma\> \le 0$ for all $\sigma \in S^+$ and consequently $\{\alpha_{n-3}, \alpha_{n-2}, \alpha_{n-1}, \alpha_n\} \inn S_{\wm}$;
\item[-] if $G$ is of type $\sE_6, \sE_7, \sE_8$ then one checks that 
$\<\alpha_2^{\vee}+\alpha_3^{\vee}+\alpha_4^{\vee}+\alpha_5^{\vee},\sigma\> \leq 0$ for all $\sigma \in S^+$ and consequently $\{\alpha_{2}, \alpha_{3}, \alpha_{4}, \alpha_5\} \inn S_{\wm}$.
\end{itemize} This proves the Lemma for groups with a simply laced Dynkin diagram. 

Next, suppose that $G$ is of type $\sC_n$ with $n\ge 3$. If $2\alpha_n \in \SNWM$, then it follows as above that $\SNWM = 2S$. On the other hand if $2\alpha_n \notin \SNWM$, then we know that $\alpha_{n-1} + \alpha_n \in \SNWM$. Then it follows from Lemma~\ref{lem:Gsat_three_simple_roots}(\ref{lemma_item_1:Gsat_three_simple_roots}) that $2\alpha_{n-2} \notin \SNWM$. Consequently $\alpha_{n-2}+\alpha_{n-1} \in \SNWM$ if $n=3$ and $\alpha_{n-2}+\alpha_{n-1} \in \SNWM$ or $\alpha_{n-3}+\alpha_{n-2} \in \SNWM$ if $n\ge 4$, since $\alpha_{n-2} \in \supp(\SNWM)$. By Lemma~\ref{lem:Gsat_three_simple_roots}(\ref{lemma_item_2:Gsat_three_simple_roots}) this implies that $\alpha_{n-2}+\alpha_{n-1} \in \SNWM$ and $\alpha_{n-3}+\alpha_{n-2} \in \SNWM$ when $n\ge 4$. Continuing to apply Lemma~\ref{lem:Gsat_three_simple_roots}(\ref{lemma_item_2:Gsat_three_simple_roots}) we deduce that $S^+ \inn \SNWM$. It then follows from  $(n-2)$ applications of Lemma~\ref{lem:Gsat_two_simple_roots}(\ref{lemma_item_3:Gsat_two_simple_roots}) that $\SNWM \cap 2S = \emptyset$ and therefore $\SNWM = S^+$. We have proven the Lemma for groups of type $\sC_n$ with $n\ge 3$

Now, we suppose that $G$ is of type $\sB_n$ with $n \ge 3$. If $2\alpha_1 \in \SNWM$, then it follows as before that $\SNWM = 2S$. If $2\alpha_1 \notin \SNWM$, then $\alpha_1 + \alpha_2 \in \SNWM$ and it follows by $(n-2)$ applications of Lemma~\ref{lem:Gsat_three_simple_roots}(\ref{lemma_item_2:Gsat_three_simple_roots}) that $S^+ \inn \SNWM$. In turn, this implies by $(n-2)$ applications of Lemma~\ref{lem:Gsat_two_simple_roots}(\ref{lemma_item_3:Gsat_two_simple_roots}) that $\SNWM \cap 2S \inn \{2\alpha_n\}$. 
Lemma~\ref{lem:Gsatsmooth_Sgamma}(\ref{lemma_item_Sgamma_A1}) rules out the possibility $\SNWM = S^+$. Indeed, one checks that $\<\alpha_{n-2}^{\vee}+\alpha_{n-1}^{\vee}+\alpha_n^{\vee}, \sigma\> \leq 0$ for all $\sigma \in S^+$. This means that if we had $\SNWM = S^+$, we would have $\{\alpha_{n-2},\alpha_{n-1},\alpha_n\} \inn S_{\wm}$, which contradicts  Lemma~\ref{lem:Gsatsmooth_Sgamma}(\ref{lemma_item_Sgamma_A1}). This shows the Lemma for groups of type $\sB_n$ with $n \ge 3$.  

Next we suppose that $G$ is of type $\sB_2$. If $2\alpha_1 \in \SNWM$, then $\alpha_1+\alpha_2 \notin \SNWM$ by Lemma~\ref{lem:Gsat_two_simple_roots}(\ref{lemma_item_3:Gsat_two_simple_roots}) and consequently $\SNWM = 2S$. On the other hand if $2\alpha_1 \notin \SNWM$, then $\alpha_1+\alpha_2 \in \SNWM$ and consequently $\SNWM =\{\alpha_1+\alpha_2\}$ or $\SNWM = \{\alpha_1+\alpha_2, 2\alpha_2\}$. Since $\sB_2 \cong \sC_2$, we have proven the Lemma for groups of type $\sB_n$ and $\sC_n$ with $n \ge 2$.

We turn to the case where $G$ is of type $\sF_4$. If $2\alpha_1 \in \SNWM$ then by Lemma~\ref{lem:Gsat_two_simple_roots}(\ref{lemma_item_3:Gsat_two_simple_roots}) and Lemma~\ref{lem:Gsat_three_simple_roots}(\ref{lemma_item_1:Gsat_three_simple_roots}) we have that $\alpha_1+\alpha_2, \alpha_2+\alpha_3 \notin \SNWM$. Since $\alpha_2 \in \supp(\SNWM)$ it follows that $2\alpha_2 \in \SNWM$ and again by Lemma~\ref{lem:Gsat_three_simple_roots}(\ref{lemma_item_1:Gsat_three_simple_roots}) that $\alpha_3 + \alpha_4 \notin \SNWM$. Consequently $\SNWM = 2S$. On the other hand, if $2\alpha_1 \notin \SNWM$, then $\alpha_1+\alpha_2 \in \SNWM$ and by Lemma~\ref{lem:Gsat_three_simple_roots}(\ref{lemma_item_2:Gsat_three_simple_roots}) $\alpha_2 + \alpha_3 \in \SNWM$. This implies by Lemma~\ref{lem:Gsat_two_simple_roots}(\ref{lemma_item_3:Gsat_two_simple_roots}) and Lemma~\ref{lem:Gsat_three_simple_roots}(\ref{lemma_item_1:Gsat_three_simple_roots}) that $2\alpha_2 \notin \SNWM$ and $2\alpha_4 \notin \SNWM$. Since $\alpha_4 \in \supp(\SNWM)$, we deduce that $\alpha_3+\alpha_4 \in \SNWM$ and from 
Lemma~\ref{lem:Gsat_two_simple_roots}(\ref{lemma_item_3:Gsat_two_simple_roots}) we obtain that $2\alpha_3 \notin \SNWM$. We have shown that if $2\alpha_1 \notin \SNWM$, then $\SNWM = S^+$, but this is impossible by Lemma~\ref{lem:Gsatsmooth_Sgamma}(\ref{lemma_item_Sgamma_A1}) since $\<\alpha_1^{\vee}+ \alpha_2^{\vee} + \alpha_3^{\vee}, \sigma\>=0$ for all $\sigma \in S^+$. This proves the Lemma for $G$ of type $\sF_4$. 

Finally, when $G$ is of type $\sG_2$, it follows from Lemma~\ref{lem:Gsat_two_simple_roots} that $\SNWM = S^+$ or $\SNWM=2S$. The former is not possible by Lemma~\ref{lem:Gsatsmooth_Sgamma}(\ref{lemma_item_Sgamma_A1}) since $\<\alpha_1^{\vee}+\alpha_2^{\vee}, \alpha_1+\alpha_2\> = 0$ and therefore $\{\alpha_1,\alpha_2\} \in S_{\wm}$. This finishes the proof of the Lemma. 
\end{proof}

\begin{proof}[Proof of Proposition~\ref{prop:MonoidsOtherTypes}] \label{proof:prop:MonoidsOtherTypes}
We begin by proving ``$\Rightarrow$''. Assume that $\wm$ is smooth. It follows from Lemma~\ref{lem:Sigma_N_other_types_full_rank} that we have to consider three possibilities for $\SNWM$. If $\SNWM = 2S$ then it follows from (\ref{item:osiginlat}) and (\ref{item:corooteven}) in  Proposition~\ref{prop:adapnsphroots} that $2\rl \inn \Z\wm \inn 2\wl$. 

The second possibility we have to consider is that $G$ is of type $\sC_n$ with $n\ge 2$ and $\SNWM = S^+$. One checks that $\<\sum_{i \text{ odd}} \alpha_i^{\vee}, \sigma\> \leq 0$ for all $\sigma \in S^+$. Consequently, $\{\alpha_i \colon i \text { odd}\} \inn S_{\wm}$. It follows from Theorem~\ref{thm:G_sat_char}(\ref{cond:partofbasis}) that 
\begin{equation}\{\alpha_i^{\vee}|_{\Z\wm} \colon i \text { odd}\} \text{ is part of a basis of } \Z\wm^*. \label{eq:oddbasis}
\end{equation} 
By Proposition~\ref{prop:adapnsphroots}(\ref{item:osiginlat}) we know that $S^+ \inn \Z\wm$. A simple computation shows that $\Z S^+ $ contains $\<\om_k \colon k \text{ even}\>_{\Z}$ and so $\{\om_k \colon k \text{ even}\} \inn \Z\wm$. Since $\{\om_k \colon k \text{ even}\}$ is part of a basis of $\wl$, it is part of a basis of $\Z\wm$. We choose $F \inn \Z\wm$ so that $F \cup \{\om_k \colon k \text{ even}\}$ is a basis of $\Z\wm$. Without loss of generality, we may assume that $F \inn \<\om_i \colon i \text{ odd}\>_{\Z}$. 
Using Lemma~\ref{lem:part_of_basis} it now follows from \eqref{eq:oddbasis} that $\Z F = \<\om_i \colon i \text{ odd}\>_{\Z}$ and therefore that $\Z\wm = \wl$.

The third and final possibility is that $G$ is of type $\sB_n$ with $n\ge 2$ and $\SNWM = S^+ \cup \{2\alpha_n\}$. Then it follows from  (\ref{item:osiginlat}) and (\ref{item:corooteven}) in  Proposition~\ref{prop:adapnsphroots} that $\Z(S^+ \cup \{2\alpha_n\}) \inn \Z\wm \inn \<\om_1,\om_2,\ldots,\om_{n-1},2\om_n\>_{\Z}$. This finishes the proof of the implication ``$\Rightarrow$.''

We now prove the other implication ``$\Leftarrow$.'' This is a straightforward application of Theorem~\ref{thm:G_sat_char}. Indeed, if $2\rl \inn \Z\wm \inn 2\wl$, then it follows from Proposition~\ref{prop:adapnsphroots} that $\SNWM = 2S$. One computes that $S_{\wm} = \emptyset$     and it follows that $\wm$ is smooth. Similarly, if $G$ is of type $\sB_n$ with $n\ge 2$ and $\<\alpha_1+\alpha_2, \alpha_2+\alpha_3, \ldots, \alpha_{n-1}+\alpha_n, 2\alpha_n\>_{\Z} \inn \Z\wm \inn \<\om_1, \om_2, \ldots,\om_{n-1},2\om_n\>_{\Z}$ then it follows from Proposition~\ref{prop:adapnsphroots} that $\SNWM = S^+ \cup \{2\alpha_n\}$. One computes that $S_{\wm} = \emptyset$ and it again follows that $\wm$ is smooth. Finally, if $G$ is of type $\sC_n$ with $n\ge 2$ and $\Z\wm = \wl$, then $\SNWM = S^+$ and one computes that $S_{\wm} = \{\alpha_i \colon i \text { odd}\}$. It follows that $(S_{\wm}, S^p(\wm), \SNWM \cap \Z S_{\wm}) = (\{\alpha_i \colon i \text { odd}\}, \emptyset, \emptyset)$, which is admissible. Condition~(\ref{cond:partofbasis}) of Theorem~\ref{thm:G_sat_char} is clearly met and condition~(\ref{cond:sumisNsphericalroot}) is trivially met. It follows once again that $\wm$ is smooth. This completes the proof of the Proposition.
\end{proof}

\section{Smooth affine spherical \texorpdfstring{$\SL(2)\times \C^{\times}$}{SL(2)xC*}-varieties} \label{sec:SL_2_times_C_star}

In this section, we will apply the results from \cite{charwm_arxivv2} to classify all smooth affine spherical $\slcst$-varieties, see Theorem~\ref{thm:SL_2_times_C_star} and Corollary~\ref{cor:sl2cst}.

Since we are no longer restricting ourselves to $G$-saturated weight monoids and varieties in this and the next section, we will apply the  general characterization of smooth weight monoids \cite[Theorem 4.2]{charwm_arxivv2}, rather than Theorem~\ref{thm:G_sat_char} above. Because stating the  general characterization requires a substantial number of additional notions from the theory of spherical varieties, we have not included it in this paper. In this section and the next, we have included precise references whenever we appeal to the results in \cite{charwm_arxivv2} in proofs.     

Without the assumption of $G$-saturatedness, the combinatorial description of the set $\SNWM$  becomes a little more involved than in Proposition~\ref{prop:adapnsphroots} because it is now possible for $\SNWM$ to contain simple roots. Before restricting ourselves to $G=\slcst$, we recall this description, which we will also use in Section~\ref{sec:woodward}, from \cite{msfwm}. To do so, we introduce some additional notation: if $\wm$ is a normal submonoid of $\dw$, then we denote
\begin{enumerate}[1.]
\item by $\wm^{\vee}$ the \textbf{dual cone} of $\wm$ in $\Hom_{\Z}(\wm, \Q)$, that is $$\wm^{\vee}:=\{v\in\Hom_{\Z}(\Z\wm,\Q):
  \langle v,\gamma\rangle\geq0\mbox{ for all }\gamma\in\wm\};$$
\item by $E(\wm)$ the set of primitive vectors on the extremal rays of $\wm^{\vee}$, that is $$E(\wm):= \{\delta \in (\Z\wm)^* \colon \delta
\text{ spans an extremal ray of }\wm^\vee \text{ and } \delta \text{ is primitive in } \Z\wm^*\};$$
\item by $a(\alpha)$ the set $\{\delta \in  (\Z\wm)^*\colon \<\delta,\alpha\>=1 \text{
  and } \bigl(\delta
\in E(\wm) \text{ or } \alpha^{\vee}|_{\Z\wm} - \delta \in E(\wm)\bigr)\}$ for $\alpha \in \sr\cap \Z\wm$. 
\end{enumerate}

\begin{proposition}[{\cite[Corollary 2.17]{msfwm}}] \label{prop:adapnsphroots_general}
Let $\wm$ be a normal submonoid of $\dw$. 
If $\sigma \in \Sigma^{sc}(G)$, then $\sigma \in \SNWM$
if and only if all of the following conditions hold:
\begin{enumerate}[(1)]
\item $\sigma \in \Z\wm $; \label{item:inasr1}
\item $\sigma$ is compatible with $S^p(\wm)$, that is
\begin{itemize}
\item[-] if $\sigma=\alpha_1+\ldots+\alpha_n$ with support of type $\mathsf{B}_n$ then
$\{\alpha_2, \alpha_3, \ldots, \alpha_{n-1}\} \inn S^p(\wm)$ and $\alpha_n \notin S^p(\wm)$;
\item[-] if $\sigma=\alpha_1+2(\alpha_2+\ldots+\alpha_{n-1})+\alpha_n$ with support of type $\mathsf{C}_n$ then
$\{\alpha_3, \alpha_4, \ldots, \alpha_n\}
\inn S^p(\wm)$; 
\item[-] if $\sigma$ is any other element of $\Sigma^{sc}(G)$ then 
\(\{\alpha \in \supp(\sigma)\colon \<\alpha^{\vee}, \sigma\> =0\}
\inn S^p(\wm)\);
\end{itemize} \label{item:inasr2}
\item if $\sigma \notin \sr$ and $\delta \in E(\wm)$ such that
  $\<\delta, \sigma\> > 0$ then there exists $\beta \in S\setminus
  S^p(\wm)$ such that $\beta^\vee|_{\Z\wm}$ is a positive multiple of
  $\delta$; \label{item:inasr3}

\item if $\sigma \in \sr$ then   \label{item:inasr4}
\begin{enumerate}[(a)]
\item $a(\sigma)$ has two elements; and    \label{item:inasr4a}
\item $\<\delta, \gamma\> \ge 0$ for all $\delta \in a(\sigma)$ and
  all $\gamma \in \wm$; and     \label{item:inasr4b}
\item $\<\delta, \sigma\> \le 1$ for all $\delta \in E(\wm)$; \label{item:inasr4c}
\end{enumerate}
\item if $\sigma = 2\alpha \in 2\sr$, then $\<\alpha^{\vee}, \gamma\>
\in 2\Z$ for all $\gamma \in \wm$; \label{item:inasr5}
\item if $\sigma = \alpha + \beta$ with $\alpha,\beta \in \sr$ and \label{nsorths}
  $\alpha \perp \beta$, then $\alpha^{\vee} = \beta^{\vee}$ on $\wm$. 
\end{enumerate}
\end{proposition}

From now on we will assume that $G=\slcst$ and we will (by abuse of notation) use $T$ for the maximal torus of $\SL(2)$ consisting of diagonal matrices. Our chosen maximal torus of $G$ will be $T\times \C^{\times}$ and as a Borel subgroup of $G$ we will take $\{\left(\begin{smallmatrix} t & a \\ 0 & t^{-1}\end{smallmatrix}\right) \colon t\in \C^{\times},a \in \C\}\times \C^{\times}$. We will use $\om$ for the fundamental weight of $\SL(2)$, which is the highest weight of the defining representation $(\SL(2),\C^2)$. Further we will use $\eps$ for the character of $\C^{\times}$ which is the identity map. The simple root of $\SL(2)$ (and of $G$) will be denoted $\alpha$. As is well known, $\alpha=2\om$. Consequently, the weight lattice, set of simple roots, root lattice and monoid of dominant weights of $G=\slcst$ are:
\[\wl = \<\om,\eps\>_{\Z}, \quad S=\{\alpha\}, \quad \rl=\Z\alpha, \quad \dw = \<\om,\eps,-\eps\>_{\N}.
\]
The simple coroot $\alpha^{\vee}$ in $\Hom_{\Z}(\wl,\Z)$ satisfies
\[\<\alpha^{\vee},\om\>=1 \text{ and } \<\alpha^{\vee}, \eps\>=0.\]

In the next lemma, we specialize Proposition~\ref{prop:adapnsphroots_general} to the case $G=\slcst$, in a form that is convenient for the proof of Theorem~\ref{thm:SL_2_times_C_star}.  
\begin{lemma} \label{lem:sl2cstSigmaN}
Let $G= \slcst$ and let $\wm$ be a normal submonoid of $\dw = \<\om,\eps,-\eps\>_{\N}$. Then the following hold:
\begin{enumerate}[(1)]
\item $\SNWM \in \{\emptyset,\{\alpha\},\{2\alpha\}\}$; \label{lem:sl2cstSigmaN:item1}
\item $\SNWM = \{2 \alpha\}$ if and only if \label{lem:sl2cstSigmaN:item2}
\begin{enumerate}[(i)]
\item $2\alpha \in \Z\wm \inn \<2\om,\eps\>_{\Z}$; and \label{lem:sl2cstSigmaN:item2a}
\item if $\rk \Z\wm = 2$ and $\Q_{\ge 0}\wm \neq \Q_{\ge 0}\dw$ then there exist $u,v,w\in \Z$ with $u>0, w \neq 0$ and $vw\ge 0$ such that $u\alpha+v\eps$ and $w\eps$ are the primitive elements of $\Z\wm$ on the two rays of $\Q_{\ge 0}\wm$; \label{lem:sl2cstSigmaN:item2b}
\end{enumerate}
\item $\SNWM = \{\alpha\}$  \label{lem:sl2cstSigmaN:item3}
if and only if there exists $\lambda \in \wm$ 
such that
\begin{enumerate} [(i)]
\item $\<\alpha^{\vee},\lambda\> > 0$; \label{lem:sl2cstSigmaN:item3a}
\item $\Z\wm = \Z\alpha \oplus \Z\lambda$; and \label{lem:sl2cstSigmaN:item3b}
\item $\Q_{\ge 0}\wm = \<\lambda,\gamma\>_{\Q_{\ge 0}}$ where  \label{lem:sl2cstSigmaN:item3c}
$\gamma=\<\alpha^{\vee},\lambda\>\alpha-\lambda$, or $\gamma=\alpha$, or $\gamma=a\alpha + b \lambda$ for some $a,b \in \N \setminus\{0\}$ with $\gcd(a,b)=1$. 
\end{enumerate}
\end{enumerate}
\end{lemma}
\begin{proof}
We begin with assertion~(\ref{lem:sl2cstSigmaN:item1}). Since $G$ has a root system of type $\sA_1$ it is clear from Definition~\ref{def:scspherrootsG} that $\Sigma^{sc}(G)=\{\alpha,2\alpha\}$. Since $\SNWM \inn \Sigma^{sc}(G)$, all we have to show is that if $\alpha \in \SNWM$ then $2\alpha \notin \SNWM$. If $\alpha \in \SNWM$ then it follows from (\ref{item:inasr4a}) in Proposition~\ref{prop:adapnsphroots_general} that there is $\delta \in E(\wm)$ with $\<\delta,\alpha\>>0$ and $\delta \notin \Q_{>0}\alpha^{\vee}|_{\Z\wm}$. By (\ref{item:inasr3}) of the same proposition this implies that $2\alpha \notin \SNWM$. 

We move to assertion~(\ref{lem:sl2cstSigmaN:item2}). Observe that (\ref{lem:sl2cstSigmaN:item2a}) holds if and only if (\ref{item:inasr1}) and (\ref{item:inasr5}) of Proposition~\ref{prop:adapnsphroots_general} hold and that (\ref{lem:sl2cstSigmaN:item2b}) is equivalent to (\ref{item:inasr3}) of that Proposition. Condition (\ref{item:inasr2}) of the Proposition is trivial in this case.  

We now prove assertion~(\ref{lem:sl2cstSigmaN:item3}). We begin by showing that the conditions (\ref{lem:sl2cstSigmaN:item3a}), (\ref{lem:sl2cstSigmaN:item3b}) and (\ref{lem:sl2cstSigmaN:item3c}) are sufficient for $\alpha$ to belong to $\SNWM$. 
Condition (\ref{item:inasr1}) of Proposition~\ref{prop:adapnsphroots_general} follows from (\ref{lem:sl2cstSigmaN:item3b}) and Condition (\ref{item:inasr2}) of the proposition is trivial. To verify condition (\ref{item:inasr4}) of the proposition, let $\{\alpha^{\#},\lambda^{\#}\}$ be the basis of $\Z\wm^*$ that is dual to the basis $\{\alpha,\lambda\}$ of $\Z\wm$. Then it is clear that $\alpha^{\#}$ is an element of $E(\wm)$ and of $a(\alpha)$. Furthermore 
\begin{equation}
\alpha^{\vee}|_{\Z\wm} = 2\alpha^{\#} + \<\alpha^{\vee},\lambda\>\lambda^{\#}
\end{equation}     
and therefore $\alpha^{\vee}|_{\Z\wm}-\alpha^{\#} = \alpha^{\#} + \<\alpha^{\vee},\lambda\>\lambda$ is a second element of $a(\alpha)$. By (\ref{lem:sl2cstSigmaN:item3c}), the set $E(\wm)$ contains two elements, and we now determine the element $\delta'$ of $E(\wm)$ that is not $\alpha^{\#}$: 
\begin{itemize}
\item[-] if $\gamma = \<\alpha^{\vee}, \lambda\>\alpha-\lambda$, then $\delta' = \alpha^{\vee}|_{\Z\wm}-\alpha^{\#}$;
\item[-] if $\gamma = \alpha$, then $\delta' = \lambda^{\#}$;
\item[-] if $\gamma = a\alpha+b\lambda$ with $a,b \in \N\setminus\{0\}$ and $\gcd(a,b)=1$ then $\delta' = -b \alpha^{\#} + a \lambda^{\#}$. 
\end{itemize}
One readily checks in all three cases that condition~(\ref{item:inasr4c}) of Proposition~\ref{prop:adapnsphroots_general} holds and that 
\begin{equation}
a(\alpha) = \{\alpha^{\#},\alpha^{\vee}|_{\Z\wm}-\alpha^{\#}\}. \label{eq:aalpha}
\end{equation}
This immediately implies condition~(\ref{item:inasr4a}) of the Proposition. Moreover, using the descriptions of $\delta'$ and $\gamma$ and condition~(\ref{lem:sl2cstSigmaN:item3a}) of the Lemma we obtain that 
\begin{equation*}
\<\alpha^{\#},\lambda\>=0, \quad \<\alpha^{\#},\gamma\> >0, \quad \<\delta',\gamma\>=0, \quad \<\delta',\lambda\>>0
\end{equation*} and condition~(\ref{item:inasr4b}) of the Proposition now follows from (\ref{lem:sl2cstSigmaN:item3c}). 

We now prove the reverse implication in assertion (\ref{lem:sl2cstSigmaN:item3}). Assume that $\alpha \in \SNWM$. It follows from (\ref{item:inasr1}) of Proposition~\ref{prop:adapnsphroots_general} and $a(\alpha) \neq \emptyset$ that $\alpha$ is primitive in $\Z\wm$. It follows from (\ref{item:inasr4a}) of the Proposition that $\rk \Z\wm = 2$  and therefore there exists $\beta \in \Z\wm$ such that $\Z\wm = \Z\alpha \oplus \Z\beta$. Again from (\ref{item:inasr4a}) of the Proposition we know that there exists $\delta \in E(\wm)$ such that 
\begin{equation}
\<\delta,\alpha\>=1 \text{ and } \alpha^{\vee} - \delta \neq \delta \label{eq:alphaveemindeltaNotdelta}
\end{equation}
Let $c:=\<\delta,\beta\>$. Then $c\alpha-\beta$ is on the line $\{\delta=0\}$, which supports a ray of $\wm$. By \eqref{eq:alphaveemindeltaNotdelta}, we know that $\delta \notin \Q_{\ge 0} \alpha^{\vee}|_{\Z\wm}$ and consequently $\alpha^{\vee}$ takes a positive value on exactly one element of $\{c\alpha-\beta, \beta-c\alpha\}$. Let $\lambda$ be that element. By construction, $\lambda$ satisfies condition (\ref{lem:sl2cstSigmaN:item3a}) of the Lemma and it is a primitive element of $\Z\wm$ on a ray of $\Q_{\ge 0}\wm$. Moreover, $\{\alpha,\lambda\}$ is also a basis of $\Z\wm$ and therefore $\lambda$ satisfies condition (\ref{lem:sl2cstSigmaN:item3b}) of the Lemma. Let ${\alpha^{\#},\lambda^{\#}}$ be the dual basis of $\Z\wm^*$. Since $\<\delta,\alpha\>=1$ and $\<\delta,\lambda\>=0$ we have $\delta=\alpha^{\#}$ and it follows from (\ref{item:inasr4a}) of Proposition~\ref{prop:adapnsphroots_general} that \eqref{eq:aalpha} holds again. 

Since $E(\wm)$ contains an element, namely $\delta$, that is not a  multiple of $\alpha^{\vee}|_{\Z\wm}$ we know that the cone $\Q_{\ge 0}\wm$ has two linearly independent rays and that $E(\wm)$ contains two elements. Let $\gamma$ be the primitive element of $\Z\wm$ on the ray of $\Q_{\ge 0}\wm$ that is not spanned by $\lambda$. Then $\gamma = a\alpha + \b\lambda$ for some $a,b \in \Z$ with $\gcd(a,b)=1$ and $a\neq 0$. Because $\<\alpha^{\#},\gamma\>=a$, it follows from  condition (\ref{item:inasr4b}) of Proposition~\ref{prop:adapnsphroots_general} that $a>0$. Let $\delta'$ be the element of $E(\wm)$ that vanishes on $\gamma$. Then $\delta' = a\lambda^{\#}-b\alpha^{\#}$. Since $\<\delta',\alpha\> = -b$ it follows  from (\ref{item:inasr4c}) of the Proposition that $b \ge -1$. We now show that $\gamma$ has to be as described in (\ref{lem:sl2cstSigmaN:item3c}). If $b > 0$ then we are done. We check the two remaining values for $b$: 
\begin{itemize}
\item[-] If $b=-1$ then $\delta' \in a(\alpha)$, and so by \eqref{eq:aalpha} we have that $\delta' = \alpha^{\vee}|_{\Z\wm}-\alpha^{\#} = \alpha^{\#}+\<\alpha^{\vee},\lambda\>\lambda^{\#}$. It follows that $a=\<\alpha^{\vee},\lambda\>$ and $\gamma=\<\alpha^{\vee},\lambda\>\alpha-\lambda$;
\item[-] If $b=0$ then $a=1$ because $\gcd(a,b)=1$ and so $\gamma=\alpha$.
\end{itemize} 
This completes the proof.  
\end{proof}

\begin{theorem} \label{thm:SL_2_times_C_star}
Let $G=\SL(2) \times \C^{\times}$. A submonoid $\wm$ of $\dw = \<\om,\eps,-\eps\>_{\N}$ is a smooth weight monoid if and only if it is one of the following submonoids of $\dw$:
\begin{enumerate}[(1)]
\item $\{0\}$;\label{list:rank0:case1}
\item $\<b\eps,-b\eps\>_{\N}$ for some $b\in \N\setminus\{0\}$;\label{list:rank1:case1}
\item $\<b\eps\>_{\N}$ for some $b \in \Z\setminus\{0\}$; \label{list:rank1:case2}
\item $\<\om + b\eps\>_{\N}$ for some $b \in \Z$;\label{list:rank1:case3}
\item $\<a\om\>_{\N}$ for some $a \in \{2,4\}$;\label{list:rank1:case4}
\item $\<\om+c\eps, b\eps, -b\eps\>_{\N}$ for some $c\in\Z, b\in \N\setminus\{0\}$ and $|c|\le\frac{b}{2}$;   \label{list:rank2:case1}
\item $\<a\om,b\eps,-b\eps\>_{\N}$ for some $a\in \{2,4\}$ and some $b\in \N\setminus\{0\}$; \label{list:rank2:case2}
\item $\<a\om, b\eps\>_{\N}$ for some $a\in \{2,4\}$ and some $b\in \Z\setminus \{0\}$; \label{list:rank2:case3}
\item $\<2\om+b\eps, 2\om-b\eps,2b\eps, -2b\eps\>_{\N}$ for some $b\in \N\setminus\{0\}$; \label{list:rank2:case4}
\item $\<2\om, a\om+b\eps\>_{\N}$ for some $a\in\N\setminus\{0\}$ and some $b \in\Z\setminus\{0\}$; \label{list:rank2:case5}
\item $\<a\om+b\eps, a\om-b\eps, 2\om\>_{\N}$ for some $a,b \in\N\setminus\{0\}$; \label{list:rank2:case6}
\item $\<2\om+b\eps, 2b\eps\>_{\N}$ for some $b\in\Z\setminus\{0\}$; \label{list:rank2:case7}
\item $\<\om+b\eps, c\eps\>_{\N}$ for some $b\in \Z$ and some $c\in\Z\setminus\{0\}$; \label{list:rank2:case8}
\item $\<4\om, 2b\eps, 2\om+b\eps\>_{\N}$ for some $b\in \Z\setminus\{0\}$. \label{list:rank2:case9}
\end{enumerate}
\end{theorem}
\begin{proof}
We first assume that $\wm$ is smooth and show that it occurs in the list of the Theorem.
We begin by assuming $\rk \Z\wm=1$. It then follows from the normality of $\wm$ that if $\wm$ is not $G$-saturated, then $\wm$ is of the form (\ref{list:rank1:case2}) in the Theorem. Hence we may assume that $\wm$ is $G$-saturated, so that we can use Theorem~\ref{thm:G_sat_char}.  If $\SNWM = \{2\alpha\}$ then it follows from Lemma~\ref{lem:sl2cstSigmaN} that $\Z\wm=\<2\om\>_{\Z}$ or $\Z\wm=\<4\om\>_{\Z}$. By the normality of $\wm$ it follows that $\wm=\<2\om\>_{\N}$ of $\wm=\<4\om\>_{\N}$. If $\SNWM=\emptyset$ then it follows from Lemma~\ref{lem:sl2cstSigmaN} that $\Z\wm=\Z\lambda$ with $\lambda \in \dw\setminus\{0,2\om,4\om\}$. Write $\lambda=a\om+b\eps$ with $a\ge 0$ and $b\in \Z$. Because $S_{\wm}=\{\alpha\}$ it follows from Theorem~\ref{thm:G_sat_char}(\ref{cond:partofbasis}) that $\alpha^{\vee}|_{\Z\wm}$ must be $0$ (when $S^p(\wm)=\{\alpha\}$) or part of a basis of $\Z\wm^*$. This implies that $a \in \{0,1\}$. If $a=0$ then $\Z\wm=\<b\eps\>_{\Z}$ and it follows from the $G$-saturatedness of $\wm$ that $\wm=\<b\eps,-b\eps\>_{\N}$. If $a=1$, then $\wm=\<\om+b\eps\>_{\N}$. We have shown that every smooth weight monoid of rank $1$ shows up in the list in the Theorem. 

We now consider the case where $\rk \Z\wm=2$. For several cases we will make use of \cite[Theorem 4.2]{charwm_arxivv2} and the notations in that paper. By Lemma~\ref{lem:sl2cstSigmaN} there are three possibilities for $\SNWM$. We first assume that $\SNWM = \{\alpha\}$. It follows from (\ref{lem:sl2cstSigmaN:item3}) of Lemma~\ref{lem:sl2cstSigmaN} that there exist $u,v \in \Z$ with $u>0$ and $v\neq 0$ such that $\Z\wm = \Z\alpha\oplus\Z\lambda$ and $\Q_{\ge 0}\wm=\Q_{\ge 0}\{\lambda,\gamma\}$ where $\lambda=u\om+v\eps$ and $\gamma$ is as in Lemma~\ref{lem:sl2cstSigmaN}(\ref{lem:sl2cstSigmaN:item3c}). Note that if $\gamma=\<\alpha^{\vee},\lambda\>\alpha-\lambda$ or $\gamma=\alpha$, then it follows from the normality of $\wm$ and Gordan's lemma (see, e.g.\ \cite[Lemma 2.9, p.52]{bruns_gubeladze_polytopes_book}) that $\wm$ is a monoid in entry (\ref{list:rank2:case6}) or entry (\ref{list:rank2:case5}) of the list in the Theorem, respectively. We claim that if $\gamma=a\alpha+b\lambda$ with $a,b \in \N\setminus\{0\}$ then $\wm$ is not smooth. Indeed, in this case $\mathcal{C}(\wm,\SNWM)$ of \cite[Definition 4.1]{charwm_arxivv2} is equal to $\wm^{\vee}$ because the element $\delta$ of $E(\wm)$ that is zero on the line spanned by $\gamma$ lies in the relative interior of the cone $\mathcal{V}(\wm,\SNWM)$ of  \cite[Definition 4.1]{charwm_arxivv2}. Consequently the $|\D(\wm,\SNWM)|$-tuple $(\rho(D))_{D \in \D(\wm,\SNWM)}$ of part (3) of \cite[Theorem 4.2]{charwm_arxivv2} is equal to
$a(\alpha) \cup \{\delta\}$ and contains $3$ elements. It can therefore not be part of a basis of $\Z\wm^*$, contradicting condition (3) of \cite[Theorem 4.2]{charwm_arxivv2} for $\wm$ to be smooth. This proves our claim and finishes the proof that if $\SNWM=\{\alpha\}$ then $\wm$ shows up in the list of the Theorem. 

Next we assume that $\SNWM=\{2\alpha\}$. Then it follows from Lemma~\ref{lem:sl2cstSigmaN}(\ref{lem:sl2cstSigmaN:item2a}) that $\alpha$ or $2\alpha$ is primitive in $\Z\wm$. From the same assertion we also obtain that if $\alpha = 2\om$ is primitive, then 
\begin{equation} \label{eq:alphaprim}
\text{there exists } c \in \Z\setminus \{0\} \text{ such that } \Z\wm =\<\alpha,c\eps\>_{\Z}, 
\end{equation}
and that if $2\alpha$ is primitive then 
\begin{equation} \label{eq:twoalphaprim}
\text{there exist } x \in \{0,1\} \text{ and } d\in \Z\setminus\{0\} \text{ such that } \Z\wm = \<2\alpha,x\alpha+d\eps\>_{\Z}.
\end{equation}
Suppose $\Q_{\ge 0}\wm = \Q_{\ge 0}\dw$. Then it follows from the normality of $\wm$ and Gordan's lemma that $\wm$ is the monoid (\ref{list:rank2:case2}) in the Theorem with $a = 2$ and $b=|c|$ if $\alpha$ is primitive. Similarly, if $2\alpha$ is primitive then $\wm$ is the monoid (\ref{list:rank2:case2}) with $a=4$ and $b=|d|$ if $x=0$, and $\wm$ is the monoid (\ref{list:rank2:case4}) with $b=|d|$ if $x=1$. Next we suppose that $\Q_{\ge 0}\wm$ is a strict subset of $\Q_{\ge 0}\dw$. Let $u,v,w$ be as in Lemma~\ref{lem:sl2cstSigmaN}(\ref{lem:sl2cstSigmaN:item2b}). 

If $\alpha$ is primitive, then the set $\Sigma^{sc}(\wm)$ defined in \cite[Definition 2.11]{charwm_arxivv2} is equal to $\{\alpha\}$, by \cite[Proposition 2.7]{charwm_arxivv2}. We claim that we must have $v=0$. Consequently $u=1$ and it follows that $\wm$ is the monoid in (\ref{list:rank2:case3}) for $a=2$ and $b=w$. The claim follows from condition (3) of \cite[Theorem 4.2]{charwm_arxivv2} for $\wm$ to be smooth. Indeed, if $v \neq 0$, then $\mathcal{C}(\wm,\Sigma^{sc}(\wm))$ of \cite[Definition 4.1]{charwm_arxivv2} is equal to $\wm^{\vee}$. The fact that $\alpha \in \Sigma^{sc}(\wm)$ then implies that $\mathcal{D}(\wm,\Sigma^{sc}(\wm))$ of \loccit\ contains $2$ elements $D^+$ and $D^-$ such that $\rho(D^+)=\rho(D^-)=\frac{1}{2}\alpha^{\vee}|_{\Z\wm}$. As a consequence, the $|\D(\wm,\Sigma^{sc}(\wm))|$-tuple $(\rho(D))_{D \in \D(\wm,\Sigma^{sc}(\wm))}$ cannot be 
part of a basis $\Z\wm^*$ which contradicts condition (3) of \cite[Theorem 4.2]{charwm_arxivv2}. 

If $2\alpha$ is primitive in $\Z\wm$ and $v=0$ then it follows that $u=2$. If $x=0$ in \eqref{eq:twoalphaprim} then $\wm$ is the monoid in (\ref{list:rank2:case3}) for $a=4$ and $b=w$. If $x=1$ in \eqref{eq:twoalphaprim}  then it follows from the fact that $w\eps $ is a primitive element of $\Z\wm$ that $w \in \{-2d,2d\}$. It follows that $\wm$ is the monoid in (\ref{list:rank2:case9}) with $b=w/2$. On the other hand, if $2\alpha$ is primitive and $v\neq 0$, 
then $\mathcal{C}(\wm,\SNWM)$ of \cite[Definition 4.1]{charwm_arxivv2} is equal to $\wm^{\vee}$  and the $|\D(\wm,\SNWM)|$-tuple $(\rho(D))_{D \in \D(\wm,\SNWM)}$ of \loccit\ is equal to $E(\wm)=\{\frac{1}{2}\alpha^{\vee}|_{\Z\wm},\delta\}$ where $\delta$ be the element of $E(\wm)$ which vanishes on the ray $u\alpha+v\eps$ of $\Q_{\ge 0}\wm$. By part (3) in \cite[Theorem 4.2]{charwm_arxivv2}, $E(\wm)$ is a basis of $\Z\wm^*$. This implies that $\{u\alpha+v\eps, w\eps\}$ is the dual basis of $\Z\wm$. Consequently $\<\frac{1}{2}\alpha^{\vee}|_{\Z\wm}, u\alpha+v\eps\>=1$ implies that $u = 1$. Moreover it follows from  part (2) in \cite[Theorem 4.2]{charwm_arxivv2} that $\overline{\soc}(\wm,\SNWM)$ is the third entry of \cite[Table 2]{charwm_arxivv2} for $n=1$ which implies that $\<\delta,2\alpha\>=-1$. Since $\<\delta,\alpha+v\eps\>=0$ this implies that $\<\delta,v\eps\>=\frac{1}{2}$ so that $v=w/2$. It follows that $\wm$ is the monoid (\ref{list:rank2:case7}) in the Theorem's list for $b=v$.  
 
The third possibility for $\SNWM$ is that it is empty. If $\Q_{\ge 0}\wm = \Q_{\ge 0}\dw$, then it follows from part (3) of \cite[Theorem 4.2]{charwm_arxivv2} that $\alpha^{\vee}|_{Z\wm}$ is part of a basis of $\Z\wm^*$. One deduces that $\wm$ is a monoid in the entry (\ref{list:rank2:case1}) of the Theorem. On the other hand, if $\Q_{\ge 0}\wm$ is a strict subset of $\Q_{\ge 0}\dw$ then it follows from \loccit\ that $\alpha^{\vee}_{\Z\wm} \in E(\wm)$ and that $E(\wm)$ is a basis of $\Z\wm^*$. It follows that $\wm$ is free and is generated (as a monoid) by the dual basis of $\Z\wm$. Consequently, $\wm$ is one of the monoids in entry (\ref{list:rank2:case8}) of the Theorem. This completes the proof of the fact that every smooth weight monoid for $G$ occurs in the list of the Theorem.  

What is left is to prove the reverse implication, namely that every monoid in the list of the Theorem is smooth. For the monoids in (\ref{list:rank0:case1}), (\ref{list:rank1:case1}), (\ref{list:rank1:case3}), (\ref{list:rank1:case4}), (\ref{list:rank2:case1}), (\ref{list:rank2:case2}) and (\ref{list:rank2:case4}) one can  use  Theorem~\ref{thm:G_sat_char} since they are $G$-saturated. For the other monoids one can use \cite[Theorem 4.2]{charwm_arxivv2}. This only involves elementary, if somewhat lengthy, verifications for each monoid, which we leave to the reader. Alternatively, one can exhibit for each monoid $\wm$ a smooth affine spherical $G$-variety $X$ with $\wm(X)=\wm$ (these varieties are given in Table~\ref{table:sl2cst}, but determining their weight monoid is left to the reader). 
\end{proof}

\begin{corollary} \label{cor:sl2cst}
Let $G =\SL(2) \times \C^{\times}$. Every smooth affine spherical $G$-variety is $G$-equivariantly isomorphic to a (unique) $G$-variety in Table~\ref{table:sl2cst}. 
\end{corollary}
\begin{proof}
By Losev's Theorem~\ref{thm:losev} and Theorem~\ref{thm:SL_2_times_C_star} it suffices to prove that each variety $X$ in Table~\ref{table:sl2cst} has the indicated weight monoid $\wm(X)$. We leave this verification to the reader. One way to proceed is to use basic facts in the representation theory of $\slcst$ to explicitly determine the highest weights of $\SL(2)\times \C^{\times}$ that occur in the coordinate ring $\C[X]$ of $X$.   
\end{proof}

\newcounter{slcstarnumber}
\newcommand{\atslcstn}{\addtocounter{slcstarnumber}{1}\theslcstarnumber}

\begin{longtable}{|c|c|c|c|c|}
\caption{Smooth affine spherical $(\SL(2)\times \C^{\times})$--varieties.}\label{table:sl2cst}\\
\hline
Case & $X$ & $\wm(X)$ & $\Sigma^N(X)$ & parameters \\
\hline
\atslcstn & point & $\{0\}$ & $\emptyset$ & \\
\hline
\atslcstn & $\C^{\times}_{b\eps}$ & $\<b\eps,-b\eps\>_{\N}$ &  $\emptyset$ &$b\in \N\setminus\{0\}$ \\
\hline
\atslcstn & $\C_{-b\eps}$  & $\<b\eps\>_{\N}$ & $\emptyset$ & $b \in \Z\setminus\{0\}$ \\
\hline
\atslcstn & $\C^2\otimes \C_{-b\eps}$ & $\<\om + b\eps\>_{\N}$ & $\emptyset$ & $b \in \Z$ \\
\hline
\atslcstn & $\SL(2)/T$  & $\<2\om\>_{\N}$ &$\{2\alpha\}$ & \\
\hline
\atslcstn & $\SL(2)/N(T)$ & $\<4\om\>_{\N}$ & $\{2\alpha\}$ &\\
\hline
\atslcstn & $(\C^2 \otimes \C_{-b\eps})\times \C^{\times}_{-c\eps}$ & $\<\om+b\eps, c\eps, -c\eps\>_{\N}$ & $\emptyset$ & \begin{tabular}{c} $b\in\Z,$\\ $c\in \N\setminus\{0\},$ {\scriptsize $|b|\le\frac{c}{2}$}\end{tabular} \\
\hline
\atslcstn & $\SL(2)/T \times \C^{\times}_{-b\eps}$ & $\<2\om,b\eps,-b\eps\>_{\N}$ & $\{2\alpha\}$ & $b\in \N\setminus\{0\}$ \\
\hline
\atslcstn & $\SL(2)/N(T) \times \C^{\times}_{-b\eps}$  &  $\<4\om,b\eps,-b\eps\>_{\N}$ & $\{2\alpha\}$ & $b\in \N\setminus\{0\}$ \\
\hline
\atslcstn & \begin{tabular}{c}$(\SL(2)\times\C^{\times})/H_b$, \\
{\scriptsize $H_b=\{ (n,\sigma(n) z) \} $ } \\
{\scriptsize where $n\in N(T)$, } \\
{\scriptsize $z\in\C^\times$, $\sigma(n) z^b=1$,} \\
{\scriptsize $\sigma$ nontriv.\ char. of $N(T)$} \end{tabular} & \begin{tabular}{c}
$\<2\om+b\eps, 2\om-b\eps,$\\ $2b\eps, -2b\eps\>_{\N}$ \end{tabular} & $\{2\alpha\}$ & $b\in \N\setminus\{0\}$ \\ 
\hline
\atslcstn & $(\C^2 \otimes \C_{-b\eps})\times \C_{-c\eps}$ & $\<\om+b\eps, c\eps\>_{\N}$ & $\emptyset$ &\begin{tabular}{c} $b\in \Z,$\\ $c\in\Z\setminus\{0\}$\end{tabular} \\
\hline
\atslcstn & $\SL(2)/T \times \C_{-b\eps}$ & $\<2\om, b\eps\>_{\N}$ & $\{2\alpha\}$ & $b\in \Z\setminus \{0\}$ \\
\hline
\atslcstn & $\SL(2)/N(T) \times \C_{-b\eps}$ & $\<4\om, b\eps\>_{\N}$ & $\{2\alpha\}$ & $b\in \Z\setminus \{0\}$ \\
\hline
\atslcstn & \begin{tabular}{c}$\SL(2)\times^{T} \C_{a\om}$,\\
{\scriptsize $\C^{\times} \text{ acts on } \C_{a\om}$} \\ {\scriptsize $\text{ with weight }-b\eps$} 
\end{tabular}
&  $\<2\om, a\om+b\eps\>_{\N}$ & $\{\alpha\}$ & \begin{tabular}{c} $a\in\N\setminus\{0\}$,\\ $b \in\Z\setminus\{0\}$\end{tabular} \\
\hline
\atslcstn & $\displaystyle \frac{\SL(2)\times\C^{\times}}{\ker(a\om - b\eps)}$ & $\<a\om+b\eps, a\om-b\eps, 2\om\>_{\N}$ & $\{\alpha\}$ &  $a,b \in\N\setminus\{0\}$ \\
\hline
\atslcstn & $S^2\C^2 \otimes \C_{-b\eps}$ & $\<2\om+b\eps, 2b\eps\>_{\N}$ & $\{2\alpha\}$ & $b\in\Z\setminus\{0\}$ \\
\hline
\atslcstn & \begin{tabular}{c}$\SL(2) \times^{N(T)}\C_{\sigma}$,\\ {\scriptsize $\sigma$ nontriv.\ char.\ of $N(T)$,}\\ {\scriptsize $\C^{\times} \text{ acts on } \C_{\sigma}$} \\ {\scriptsize $\text{ with weight }-b\eps$}\end{tabular} & $\<4\om, 2b\eps, 2\om+b\eps\>_{\N}$ & $\{2\alpha\}$ &  $b\in \Z\setminus\{0\}$ \\
\hline
\end{longtable}

\begin{remark}\label{rem:altern}
As already mentioned in the introduction, our results can also be
deduced using other techniques. For example, Corollary~\ref{cor:sl2cst}
and Table~\ref{table:sl2cst} could be obtained by using the approach of
\cite{knop&bvs-classif}, taking as a starting point the fact that a
smooth affine spherical $\slcst$-variety is equivariantly isomorphic to a
homogeneous vector bundle of the form $(\slcst)\times^H V$, where $H$ is
a reductive subgroup of $\slcst$ and $V$ is an $H$-module. The pairs
$(H,V)$ for which $(\slcst)\times^H V$ is spherical can be classified
using a careful case-by-case analysis.
\end{remark}

\begin{remark} 
In Table~\ref{table:sl2cst_cones}, we have listed all the cones $\Q_{\ge 0}\wm$, where $\wm$ runs through the weight monoids of smooth affine spherical $\SL(2)\times \C^{\times}$--varieties such that $\Q_{\ge 0}\wm$ is pointed. Similar pictures of $\Q_{\ge 0}\wm$ were also included in F.~Knop's 2012 lecture on ``Multiplicity free Hamiltonian actions'' in the ``Lie-Theorie und komplexe Geometrie'' Seminar at the Universit\"at Marburg. In view of the conditions~\eqref{eq:sphericity}, the table gives the possible shapes of the moment polytope $\mathcal{P}$ of a multiplicity free Hamiltonian manifold near a vertex $a \in \mathcal{P}$ such that $G(a)$ is equal (or isogenous) to $\SL(2)\times \C^{\times}$.
\end{remark}

\begin{longtable}{|c|c|c|c|c|}
\caption{Pointed cones of smooth affine spherical $\SL(2)\times
\C^{\times}$--varieties {\scriptsize (the given value of $\tan \varphi$
assumes that $\om$ and $\eps$ have the same
length.)}}\label{table:sl2cst_cones}\\
\hline
\begin{tikzpicture}[line cap=round,line join=round,>=triangle
45,x=0.2cm,y=0.2cm]
\clip(-7.,-1.) rectangle (7.,7.);
\draw (-6.,0.)-- (6.,0.);
\draw [->] (0.,0.) -- (0.,4.);
\draw [->] (0.,0.) -- (4.,0.);
\draw[color=black] (4.,.75) node {$\varepsilon$};
\draw [line width=1.pt,color=sqsqsq] (0.,0.)-- (6.,6.);
\draw [shift={(0.,0.)}]
plot[domain=0.:0.7853981633974483,variable=\t]({1.*2.*cos(\t
r)+0.*2.*sin(\t r)},{0.*2.*cos(\t r)+1.*2.*sin(\t r)});
\begin{scriptsize}
\draw[color=black] (-0.99,3.75) node {$\omega$};
\draw[color=black] (2.38,1.07) node {$\varphi$};
\end{scriptsize}
\end{tikzpicture} & \begin{tikzpicture}[line cap=round,line
join=round,>=triangle 45,x=0.2cm,y=0.2cm]
\clip(-7.,-1.) rectangle (7.,7.);
\fill[color=sqsqsq,fill=sqsqsq,fill opacity=0.10000000149011612]
(-6.,6.) -- (0.,0.) -- (6.,6.) -- cycle;
\draw (-6.,0.)-- (6.,0.);
\draw [->] (0.,0.) -- (0.,4.);
\draw [->] (0.,0.) -- (4.,0.);
\draw[color=black] (4.,.75) node {$\varepsilon$};
\draw [color=sqsqsq] (-6.,6.)-- (0.,0.);
\draw [color=sqsqsq] (0.,0.)-- (6.,6.);
\draw [shift={(0.,0.)}]
plot[domain=0.7853981633974483:2.356194490192345,variable=\t]({1.*2.1354624791833734*cos(\t
r)+0.*2.1354624791833734*sin(\t r)},{0.*2.1354624791833734*cos(\t
r)+1.*2.1354624791833734*sin(\t r)});
\draw [shift={(0.,0.)}]
plot[domain=0.7853981633974483:2.356194490192345,variable=\t]({1.*2.1354624791833734*cos(\t
r)+0.*2.1354624791833734*sin(\t r)},{0.*2.1354624791833734*cos(\t
r)+1.*2.1354624791833734*sin(\t r)});
\begin{scriptsize}
\draw[color=black] (-0.99,3.75) node {$\omega$};
\draw[color=black] (1.28,2.5) node {$\varphi$};
\draw[color=black] (-1.28,2.5) node {$\varphi$};
\end{scriptsize}
\end{tikzpicture} & \begin{tikzpicture}[line cap=round,line
join=round,>=triangle 45,x=0.2cm,y=0.2cm]
\clip(-7.,-1.) rectangle (7.,7.);
\fill[color=sqsqsq,fill=sqsqsq,fill opacity=0.10000000149011612] (0.,6.)
-- (0.,0.) -- (6.,6.) -- cycle;
\draw (-6.,0.)-- (6.,0.);
\draw [->] (0.,0.) -- (0.,4.);
\draw [->] (0.,0.) -- (4.,0.);
\draw[color=black] (4.,.75) node {$\varepsilon$};
\draw [color=sqsqsq] (0.,6.)-- (0.,0.);
\draw [color=sqsqsq] (0.,0.)-- (6.,6.);
\draw [shift={(0.,0.)}]
plot[domain=0.7853981633974483:1.5707963267948966,variable=\t]({1.*2.1354624791833734*cos(\t
r)+0.*2.1354624791833734*sin(\t r)},{0.*2.1354624791833734*cos(\t
r)+1.*2.1354624791833734*sin(\t r)});
\begin{scriptsize}
\draw[color=black] (-0.99,3.75) node {$\omega$};
\draw[color=black] (1.34,2.5) node {$\varphi$};
\end{scriptsize}
\end{tikzpicture} & \begin{tikzpicture}[line cap=round,line
join=round,>=triangle 45,x=0.2cm,y=0.2cm]
\clip(-7.,-1.) rectangle (7.,7.);
\fill[color=sqsqsq,fill=sqsqsq,fill opacity=0.10000000149011612] (6.,6.)
-- (0.,0.) -- (6.,0.) -- cycle;
\draw (-6.,0.)-- (6.,0.);
\draw [->] (0.,0.) -- (0.,4.);
\draw [->] (0.,0.) -- (4.,0.);
\draw[color=black] (4.,.75) node {$\varepsilon$};
\draw [shift={(0.,0.)}]
plot[domain=0.7853981633974483:1.5707963267948966,variable=\t]({1.*2.1354624791833734*cos(\t
r)+0.*2.1354624791833734*sin(\t r)},{0.*2.1354624791833734*cos(\t
r)+1.*2.1354624791833734*sin(\t r)});
\draw [color=sqsqsq] (6.,6.)-- (0.,0.);
\draw [color=sqsqsq] (0.,0.)-- (6.,0.);
\begin{scriptsize}
\draw[color=black] (-0.99,3.75) node {$\omega$};
\draw[color=black] (1.34,2.27) node {$\varphi$};
\end{scriptsize}
\end{tikzpicture} &
\begin{tikzpicture}[line cap=round,line join=round,>=triangle
45,x=0.2cm,y=0.2cm]
\clip(-7.,-1.) rectangle (7.,7.);
\fill[color=sqsqsq,fill=sqsqsq,fill opacity=0.10000000149011612] (6.,6.)
-- (0.,0.) -- (-6.,0.) -- (-6.,6.) --  cycle;
\draw (-6.,0.)-- (6.,0.);
\draw [->] (0.,0.) -- (0.,4.);
\draw [->] (0.,0.) -- (4.,0.);
\draw[color=black] (4.,.75) node {$\varepsilon$};
\draw [shift={(0.,0.)}]
plot[domain=0.7853981633974483:1.5707963267948966,variable=\t]({1.*2.1354624791833734*cos(\t
r)+0.*2.1354624791833734*sin(\t r)},{0.*2.1354624791833734*cos(\t
r)+1.*2.1354624791833734*sin(\t r)});
\draw [color=sqsqsq] (6.,6.)-- (0.,0.);
\draw [color=sqsqsq] (0.,0.)-- (6.,0.);
\begin{scriptsize}
\draw[color=black] (-0.99,3.75) node {$\omega$};
\draw[color=black] (1.34,2.27) node {$\varphi$};
\end{scriptsize}
\end{tikzpicture} \\
\hline
{\scriptsize \begin{tabular}{c}$0\le \varphi \le \pi$\\
$(\tan \varphi)^{-1} \in \Z\cup\{\pm\infty\}$\end{tabular}} &
{\scriptsize \begin{tabular}{c} $0<\varphi < \frac{\pi}{2}$ \\ $\tan
\varphi \in \Q$\end{tabular}}  & {\scriptsize \begin{tabular}{c}
$-\frac{\pi}{2}\le \varphi \le \frac{\pi}{2}, \varphi \ne 0$ \\ $\tan
\varphi \in \Q\cup\{\pm\infty\}$\end{tabular}} & {\scriptsize
\begin{tabular}{c}$-\pi/2 < \varphi < \pi/2$ \\$2\tan \varphi \in \Z$
\end{tabular}} &{\scriptsize \begin{tabular}{c}$-\pi/2 < \varphi <
\pi/2$ \\$2\tan \varphi \in \Z$ \end{tabular}}
\\
\hline
\end{longtable}

\begin{example} \label{ex:SU3manifold}
In this example, we use the classifications in Section \ref{sec:G_sat_full_rank} and in this section to exhibit two pairs $(\mathcal{P},\wl_0)$ with the same polytope $\P$ satisfying the conditions \eqref{eq:sphericity} for $K=\SU(3)$. As explained in the introduction, this implies by \cite[Theorem 11.2]{knop-autoHam} that there exists a multiplicity free Hamiltonian $\SU(3)$-manifold $M$ with $(\mathcal{P}_M,\wl_M) = (\mathcal{P},\wl_0)$. The complexification of $K$ is $G=\SL(3)$. We will use $\om_1,\om_2$ for the fundamental weights of $G$. We then have $\wl = \<\om_1,\om_2\>_{\Z}$. We identify the positive Weyl chamber $\ft^+$ with $\<\om_1,\om_2\>_{\R_{\ge 0}}$. Let $\mathcal{P}$ be the triangle in $\ft^+$ with vertices $0$, $\om_1$ and $\om_2$:
\[\P:=\operatorname{conv}(0,\om_1, \om_2) \inn \ft^+.\]

We need to check the conditions \eqref{eq:sphericity} at $a=0$, at $a=\om_1$ and at $a=\om_2$. Because there is an (outer) automorphism of $G$ that exchanges $\om_1$ and $\om_2$, it suffices to check them at $0$ and $\om_1$. Using the notation of the introduction we have that 
$G(0) = G$ and $
G(\om_1) = \left\lbrace\left(\begin{smallmatrix}\det(A)^{-1} & 0 \\ 0 & A \end{smallmatrix}\right) \colon A \in \GL(2)\right\rbrace \cong \GL(2)$.
Observe that the tangent cone $C_0\P$ to $\P$ at $0$ is $\ft^+$. 
To take advantage of our classification of $(\slcst)$-varieties, we will use the following isogeny
\(\varphi \colon \slcst \to G(\om_1)\colon (A,z) \mapsto \left(\begin{smallmatrix}z^{-2} & 0 \\ 0 & zA \end{smallmatrix}\right)\).
The induced map $\varphi^*$ from the weight lattice $\wl$ of $G(\om_1)$ (and of $G$) to the weight lattice $\<\om,\eps\>_{\Z}$ of $\slcst$ satisfies:
$\varphi^*(\om_1)=-2\eps$ and $\varphi^*(\om_2) = \om-\eps$.
One immediately computes that the tangent cone to $\mathcal{P}$ at $\om_1$ is
$C_{\om_1}\mathcal{P} = \<-\om_1,\om_2-\om_1\>_{\R \ge 0}.$

We check that the pair 
\[(\P,\wl_0) \text{ with } \wl_0 = \wl = \<\om_1,\om_2\>_{\Z}\] satisfies \eqref{eq:sphericity} at $a=0$ and at $a=\om_1$. From Table~\ref{table:SL}, we see that $C_0\P\cap \wl_0$ is the weight monoid of the smooth affine spherical $G(0)$-variety $\SL(3)/\Sp(2)$. At $a=\om_1$, we have that $
C_{\om_1}\mathcal{P} \cap \wl_0 = \<-\om_1,\om_2-\om_1\>_{\N}$. Consequently, $\varphi^*(C_{\om_1}\mathcal{P} \cap \wl_0) = \<2\eps,\om+\eps\>_{\N}$, which is the weight monoid of the smooth affine spherical $(\slcst)$-variety $\C^2\otimes\C_{-\eps} \times \C_{-2\eps}$, cf.\ Table~\ref{table:sl2cst}. Similarly, one can check that the pair 
\[(\P,\wl_0) \text{ with } \wl_0 = 2\wl = \<2\om_1,2\om_2\>_{\Z}\] satisfies \eqref{eq:sphericity} at $a=0$ and at $a=\om_1$. Indeed, $C_0\P\cap \wl_0$ is the weight monoid of $\SL(3)/\SO(3)$, while  $\varphi^*(C_{\om_1}\mathcal{P} \cap \wl_0) = \<4\eps,2\om+2\eps\>_{\N}$ is the weight monoid of $S^2\C^2 \otimes \C_{-2\eps}$.  
\end{example}

\section{Woodward's reflective polytopes} \label{sec:woodward}

In this section we study a special class of weight monoids that are not $G$-saturated. They are related to Woodward's work \cite{woodward-classif} on {\em transversal} multiplicity free actions, which correspond to reflective polytopes (see Definition~\ref{def:reflective_polytope}).

As in the introduction, $K$ is a compact connected Lie group with complexification $G$. Let $W=N(T)/T$ be the Weyl group of $G$ and of $K$. We recall that $\ft^+ \inn \Lie(T_{\R})^* \equiv \wl\otimes_{\Z}\R$ is the chosen positive Weyl chamber of $K$ corresponding to the choice of the Borel subgroup $B$ of $G$. 
\begin{definition} \label{def:reflective_polytope}
Let $\P$ be a convex polytope in $\ft^+$. 
\begin{enumerate}[1.]
\item We call $\P$ \textbf{reflective} if the following conditions are fulfilled: \label{def:refl}
\begin{enumerate}[(a)]
  \item $\P$ is of maximal dimension;
  \item for all $a\in \P$, the set of hyperplanes generated by the faces of $\P$ of codimension $1$ containing $a$ is stable under $W_a = \{w \in W\colon w\cdot a = a\}$;
  \item for all $a\in \P$, any face of $\P$ of codimension $1$ containing $a$ meets the relative interior of $\ft^+$.
 \end{enumerate}
\item Let $\wl_0$ be a sublattice of $\wl$. We say that $\P$ is \textbf{Delzant} (with respect to $\wl_0$) if for every vertex $a$ of $\P$ there exists a basis $\lambda_1,\lambda_2,\ldots,\lambda_n$ of $\wl_0$ such that $C_a\P = \<\lambda_1,\lambda_2,\ldots,\lambda_n\>_{\R_{\ge 0}}$, where $C_a\P$ is the tangent cone to $\P$ at $a$.  
\end{enumerate} 
\end{definition}

Part of \cite[Theorem 1.3]{woodward-classif} can be stated as follows. 
\begin{theorem}[Woodward] \label{thm:woodward}
Let $K$ be a compact connected Lie group. Let $\P$ be a convex polytope in $\ft^+$. 
If $\P$ is reflective and Delzant with respect to the weight lattice $\wl$ of $K$, and if the semisimple part of the stabilizer $K_x$ is simply connected for every $x \in \P$, then there exists a multiplicity free Hamiltonian manifold $M$ with trivial generic isotropy group and moment polytope $\P$. 
\end{theorem} 
Observe that the condition on $K_x$ in this theorem only needs to be checked for finitely many points $x \in \P$ since $K_x$ only depends on the face of $\ft^+$ that contains $x$ in its relative interior. 
We emphasize that Theorem~\ref{thm:woodward} is only part of Woodward's result: \loccit\ also  established the uniqueness of $M$ and characterized the multiplicity free actions that have moment polytopes which are reflective and Delzant: they are the so-called \emph{torsion-free} and \emph{transversal} multiplicity free actions. Woodward's approach is completely symplectic. 

The goal of this section is to give a proof of Theorem~\ref{thm:woodward} using the techniques we have seen in the rest of the paper, see Corollary~\ref{cor:ourwoodward} below. The property of being reflective, for a polytope $\P$, can be restated locally in terms of the weight monoids related to $\P$ as recalled in the introduction.

\begin{definition}\label{def:reflective}
Let $W$ be the Weyl group of $G$. 
A submonoid $\wm$ of $\Lambda^+$ is \textbf{reflective} if it is normal and the following hold:
 \begin{enumerate}
  \item $\wm$ has full rank;
  \item the set of hyperplanes generated by the faces of codimension $1$ of the cone spanned by $\wm$ is stable under $W$;
  \item\label{def:reflective:open} every face of codimension $1$ of the cone spanned by $\wm$ meets the open positive Weyl chamber.
 \end{enumerate}
\end{definition}

\begin{remark} \label{rem:reflnocorootwall}
If a monoid $\wm$ is reflective, then $E(\wm)$ does not contain any element proportional to $\alpha^\vee|_{\Z\Gamma}$ for $\alpha$ a simple root, thanks to property (\ref{def:reflective:open}) of the above definition. Since for a $G$-saturated submonoid $\wm$ of $\dw$ every element of $E(\wm)$ is proportional to $\alpha^\vee|_{\Z\Gamma}$ for some $\alpha \in S$, it follows in particular that a reflective submonoid of $\dw$ is never $G$-saturated when $G$ is not abelian. 
\end{remark}

We start by showing that reflectivity of $\wm$ has strong consequences on $\Sigma^{N}(\wm)$.

\begin{lemma}\label{lemma:SigmaisinS}
 Let $\wm$ be reflective. Then $\Sigma^N(\G) \subset S$.
\end{lemma}

\begin{proof}
We assume for the sake of contradiction that there exists $\sigma \in \Sigma^N(\G)\setminus S$. We claim that there exists a simple root $\alpha$ such that $\<\alpha^\vee,\sigma\> >0$. Indeed, the set $S\cup \{\sigma\}$ is linearly dependent and lies in an open half-space, whence it contains two elements forming an acute angle. One of the two elements must be $\sigma$, and our claim is proved.

Since $\Gamma\subset \Lambda^+$, the convex cone $\wm^\vee$ contains all simple coroots. Then the existence of $\alpha$ implies that there exists $\delta\in E(\Gamma)$ such that $\<\delta,\sigma\> >0$. By Proposition~\ref{prop:adapnsphroots_general}(\ref{item:inasr3}), this implies that $\delta$ is a positive multiple of a simple coroot, contradicting the assumption that $\wm$ is reflective, by Remark~\ref{rem:reflnocorootwall}.
\end{proof}

The property of being Delzant, for a polytope $\P$, can also be restated locally in terms of  weight monoids. Recall from Definition~\ref{def:reflective_polytope} that $\P$ is Delzant with respect to a sublattice $\wl_0$ of $\wl$ if and only if for every vertex $a$ of $\P$ there exists a basis $F_a$ of $\wl_0$ such that $\mathcal{C}_a\P = \R_{\ge 0}F_a$. Now, if we consider a general polytope $\P$ in $\ft^+$ and set $\wm_a=\mathcal{C}_a\P \cap \wl_0$ for any vertex $a$ of $\P$, then $\P$ is Delzant if and only if $\dim \P = \operatorname{rank} \wl_0$ and for every vertex $a$ of $\P$  
\begin{equation}
\text{there exists a basis $F_a$ of $\wl_0$ such that }\wm_a = \N F_a. \label{eq:Delzant_monoid} 
\end{equation}
Moreover, \eqref{eq:Delzant_monoid} holds if and only if
\begin{equation}
\Z\wm_a = \wl_0 \text{ and $E(\wm_a)$ is a basis of $\wl_0^*$.} 
\end{equation}
\begin{lemma}\label{lemma:SigmaisS}
Let $\wm$ be reflective. Suppose that $S\subset \Z\wm$, that $E(\wm)$ is part of a basis of $\Z\wm^*$, and that $\Z\wm$ is $W$-invariant. Then $\Sigma^N(\wm)=S$.
\end{lemma}

\begin{proof}
After Lemma~\ref{lemma:SigmaisinS}, it remains to show the inclusion $\Sigma^N(\G)\supset S$. Let $\a\in S$. To show that $\alpha\in\Sigma^N(X)$ we use Proposition~\ref{prop:adapnsphroots_general}, so we have to check the conditions required by that Proposition. 

Condition~(\ref{item:inasr1}), i.e.\ that $\alpha \in \Z\wm $, follows from our assumption $S\subset\Z\G$. Condition~(\ref{item:inasr2}) is met because $\{\beta \in \supp(\alpha)\colon \<\beta^{\vee},\alpha\>=0\} = \emptyset$.  We turn to the conditions~(\ref{item:inasr4}). We begin by proving that $a(\alpha)$ has two elements. By construction, $\wm\subset\Lambda^+$, hence $\alpha^\vee|_{\Z\wm}\in\wm^\vee$. Since $\<\alpha^{\vee},\alpha\>=2$, there exists $\delta \in E(\G)$ such that $\< \delta,\alpha\> >0$.
We show that $-s_\a\d\in\G^\vee$. Either $-s_\a\d$ or $s_\a\d$ lies on an extremal ray of $\G^\vee$, by reflectivity. Assume it is $s_\alpha\delta$, and consider the equality
\[
\la \d,\a\ra \a^\vee|_{\Z\wm}=\d-s_\a\d.
\]
It expresses $\alpha^\vee|_{\Z\wm}\in\Gamma^\vee$ as a linear combination of elements of $E(\Gamma)$, with both positive and negative coefficients. This contradicts the assumption that $E(\Gamma)$ is part of a basis, and yields that $-s_\a\d$ lies on an extremal ray of $\Gamma^\vee$.
Observe that $\delta$ is primitive in $(\Z\Gamma)^*$ by assumption, and $\Z\Gamma$ is $W$-invariant. This shows that $-s_\alpha\delta$ is also primitive in $(\Z\Gamma)^*$, hence $-s_\alpha\delta\in E(\Gamma)$.

Let us show now that $\<\delta,\alpha\>=1$. Otherwise, the element $\delta - s_\alpha\delta$ is not primitive in $(\Z\Gamma)^*$. But since $\{\delta,-s_\alpha\delta\}$ is part of a basis, then also $\{\delta,\delta-s_\alpha\delta\}$ is part of a basis. Therefore $\delta-s_\alpha\delta$ is primitive, and we conclude that $\<\delta,\alpha\>=1$.

At this point $\delta,-s_\alpha\delta$ are two elements of $E(\Gamma)$, both taking value $1$ on $\alpha$. They are distinct, otherwise we would have
\[
\alpha^\vee|_{\Z\wm} = 2\delta
\]
which contradicts the fact that $\delta$ is not a positive multiple of any simple coroot by Remark~\ref{rem:reflnocorootwall}. We deduce that both belong to $a(\alpha)$, hence $|a(\alpha)|\geq 2$.

We show that $|a(\alpha)|= 2$. If $|a(\alpha)|>2$, then there exists $\epsilon\in a(\alpha)\cap E(\Gamma)$ such that $\epsilon\notin\{\delta,-s_\alpha\delta\}$. We have $\<\epsilon,\alpha\>=1$, and by the above argument applied to $\epsilon$ instead of $\delta$, we have that $-s_\alpha\epsilon\in E(\Gamma)$. Then
\[
\alpha^\vee|_{\Z\wm} = \delta - s_\alpha\delta = \epsilon - s_\alpha\epsilon
\]
contradicts the assumed linear independence of $E(\Gamma)$. Therefore $|a(\alpha)|=2$, and we have verified condition~(\ref{item:inasr4a}) of  the Proposition.

We have also proved that $a(\alpha)\subset E(\Gamma)$, which implies in particular that $\<\delta, \gamma\> \ge 0$ for all $\delta \in a(\alpha)$ and all $\gamma \in \wm$, which is condition~(\ref{item:inasr4b}). 

Finally, it remains to check the last condition required by the Proposition, which is that $\<\delta, \alpha\> \le 1$ for all $\delta \in E(\wm)$. We have already proved in general that if $\delta\in E(\Gamma)$ satisfies $\<\delta,\alpha\> >0$, then $\<\delta,\alpha\>=1$, and this completes the proof.
\end{proof}

We single out the following already noticed consequence of the above proof.

\begin{corollary}\label{cor:SigmaisS}
Under the assumptions of Lemma~\ref{lemma:SigmaisS}, let $\alpha\in S$. Then $a(\alpha)$ consists of two elements $\delta_1,\delta_2 \in E(\wm)$ such that $\alpha^{\vee}|_{\Z\wm}=\delta_1+\delta_2$, and if $\delta\in E(\Gamma)$ satisfies $\<\delta,\alpha\> >0$, then $\<\delta,\alpha\>=1$.
\end{corollary}

\begin{lemma}\label{lemma:reflectiveadapted}
Under the assumptions of Lemma~\ref{lemma:SigmaisS}, there exists an affine spherical variety $X$ with $\Sigma^N(X)=\Sigma^N(\Gamma)$ and $\Gamma(X)=\Gamma$.
\end{lemma}

\begin{proof}
Thanks to \cite[Proposition~2.7]{charwm_arxivv2} we have to check that the condition (b) given in \cite[Proposition~2.24]{charwm_arxivv2} holds, i.e.\ that if $\alpha \in S \cap \Sigma^N(\Gamma)$, $\delta \in a(\alpha)$ and $\gamma \in \Sigma^N(\Gamma)$ satisfy $\<\delta,\gamma\> >0$, then $\gamma \in S$ and $\delta \in a(\gamma)$.
Let $\alpha, \delta$ and $\gamma$ be as above. Then $\gamma\in S$, $\delta\in E(\Gamma)$, and $\<\delta,\gamma\>=1$ by Corollary~\ref{cor:SigmaisS}. It follows that $\delta\in a(\gamma)$, and the proof is complete.
\end{proof}

\begin{lemma}\label{lemma:reflectivegeq0}
Under the assumptions of Lemma~\ref{lemma:SigmaisS}, let $\d\in E(\G)$ and $\a\in S$. Then $\la \d,\a \ra \ge 0$.
\end{lemma}
\begin{proof}
For the sake of contradiction, we assume $\la \d,\a\ra <0$. By reflectivity, we have that $-s_\alpha\delta$ is a multiple of an element of $E(\Gamma)$. Then
 \[
  \a^\vee|_{\Z\wm} = \frac{1}{\la \d,\a\ra}\delta +\left( \frac{1}{\la\d,\a\ra}\right)\left(-s_\a\d\right)
 \]
is an expression of $\a^\vee|_{\Z\wm}$ as a linear combination of elements of $E(\G)$  with at least one negative coefficient. On the other hand, $\a^\vee|_{\Z\wm}$ can be written as a linear combination of elements of $E(\G)$ with non-negative rational coefficients. This contradicts the assumption that $E(\Gamma)$ is part of a basis of $(\Z\wm)^*$.
\end{proof}

We are ready to prove the main result of this section.

\begin{theorem}\label{thm:ourwoodward}
Let $\G$ be reflective. Suppose that $S\subset \Z\G$, that $E(\G)$ is part of a basis of $(\Z\G)^*$, and that $\Z\Gamma$ is $W$-invariant. Then $\G$ is smooth, and the semisimple type of $G$ is $(\mathsf A_1)^r$ for some $r\in \N$.
\end{theorem}

\begin{proof}
Let us first prove the claim about the semisimple type of $G$. Assume $\a,\b \in S$ with $\a \ne \b$. Then, by Corollary~\ref{cor:SigmaisS}, we have
\[
 \a^\vee|_{\Z\wm}=\d_1+\d_2
\]
where $\d_1,\d_2\in E(\G)$, which yields
\[
 \la \a^\vee,\b\ra = \la \d_1, \b\ra + \la \d_2,\beta\ra.
\]
As both summands are non-negative by Lemma~\ref{lemma:reflectivegeq0}, we have $\la\a^\vee,\b\ra=0$. 

Let us prove that $\Gamma$ is smooth. We use \cite[Theorem~4.2]{charwm_arxivv2} with $\Sigma$ of \loccit\ equal to $\Sigma^N(\Gamma)$. Notice that this choice of $\Sigma$ is allowed by Lemma~\ref{lemma:reflectiveadapted}.

The objects defined in \cite[Definition~4.1]{charwm_arxivv2} are then the following.

\begin{itemize}
\item The set
\[
\A(\Gamma,\Sigma^N(\Gamma)) = \bigcup_{\alpha\in S} a(\alpha).
\]
\item The triple $\mathscr S(\Gamma,\Sigma^N(\Gamma))=(\emptyset, S, \A(\Gamma,\Sigma^N(\Gamma)))$.
\item The set $\mathcal V(\Gamma,\Sigma^N(\Gamma)) = \{v\in\Hom_\Z(\Z\Gamma,\Q): \<v,\sigma\>\leq 0\;\forall\sigma\in\Sigma^N(\Gamma)\}$.
\item The set $\Delta(\Gamma,\Sigma^N(\Gamma))=\A(\Gamma,\Sigma^N(\Gamma))$.
\item The set $\mathcal{C}(\Gamma,\Sigma^N(\Gamma))$ is the maximal face of $\Gamma^\vee$ of which the relative interior meets $\mathcal{V}(\Gamma,\Sigma^N(\Gamma))$. Since $\Sigma^N(\Gamma)=S$, any element of $\mathcal{C}(\Gamma,\Sigma^N(\Gamma))$ is zero on $S$.
\item The set $\mathcal{F}(\G,\Sigma^N(\Gamma)):=\{\delta\in\Delta(\G,\Sigma^N(\Gamma)):\delta \in \mathcal{C}(\G,\Sigma^N(\Gamma))\}$, which is empty since for all $\alpha\in S$ any element of $a(\alpha)$ takes value $1$ on $\alpha$.
\item The set $\mathcal{B}(\G, \Sigma^N(\Gamma))$ is the set of primitive elements of the lattice $(\Z\G)^*$ on extremal rays of $\mathcal{C}(\G,\Sigma^N(\Gamma))$ that do not contain any element of $\mathcal{F}(\G,\SNWM)$. This is equal to $\{\d \in E(\G): \delta|_S =0\}$.
\item The set $\mathcal{D}(\G, \Sigma^N(\Gamma))=\mathcal{B}(\G, \Sigma^N(\Gamma))\cup \mathcal{F}(\G,\Sigma^N(\Gamma))$, it is hence equal to $\mathcal{B}(\G, \Sigma^N(\Gamma))$.
\item The set $S(\G,\Sigma^N(\Gamma))$ is the set of the elements $\alpha\in \Sigma^N(\Gamma)$ such that no element in $\Delta(\Gamma,\Sigma^N(\Gamma))\setminus\mathcal{F}(\G,\Sigma^N(\Gamma))$ is in $a(\alpha)$. Then $S(\G,\Sigma^N(\Gamma))=\emptyset$.
\item The sextuple $\operatorname{soc}(\Gamma,\Sigma^N(\Gamma))=(\emptyset, S, \A(\G,\Sigma^N(\Gamma)), \emptyset, \mathcal{D}(\G,\Sigma^N(\Gamma)), \rho')$, where the map $\rho'\colon \A(\Gamma,\Sigma^N(\Gamma))\to (\Z\Sigma^N(\Gamma))^*$ is the restriction to the sublattice $\Z\Sigma^N(\Gamma)\subset \Z\Gamma$.
\item The sextuple $\overline{\operatorname{soc}}(\G,\Sigma^N(\Gamma))=(\emptyset, \emptyset, \emptyset, \emptyset, \mathcal{D}(\G,\Sigma^N(\Gamma)), \emptyset)$.
 \end{itemize}
The sextuple $\overline{\operatorname{soc}}(\G,\Sigma^N(\Gamma))$ is a product (in the sense of \cite[Section~3]{charwm_arxivv2}) of the first two entries of Table~2 in \cite{charwm_arxivv2}, which yields condition (2) of \cite[Theorem~4.2]{charwm_arxivv2}. Condition (1) of \cite[Theorem~4.2]{charwm_arxivv2} follows from Lemma~\ref{lemma:SigmaisS}. Condition (3) is fulfilled because $\mathcal{D}(\Gamma,\Sigma^N(\Gamma))$ is a subset of $E(\Gamma)$, and the proof is complete.
\end{proof}

Combining Theorem~\ref{thm:ourwoodward} with \cite[Theorem 11.2]{knop-autoHam} we obtain the following:
\begin{corollary} \label{cor:ourwoodward}
Let $K$ be a compact Lie group with weight lattice $\wl$, positive Weyl chamber $\ft^+$, Weyl group $W$ and set of simple roots $S$. Let $\wl_0$ be a sublattice of $\wl$ and $\P$ a convex polytope in $\ft^+$. If $\wl_0$ is $W_a$-invariant and contains $\{\alpha \in S\colon \<\alpha^{\vee},a\>=0\}$ for every vertex $a$ of $\P$, and if $\P$ is reflective and Delzant with respect to $\wl_0$, then there exists a multiplicity free Hamiltonian $K$-manifold $M$ such that $(\P_M,\wl_M)=(\P,\wl_0)$. 
\end{corollary}

\begin{proof}
As in the introduction, we denote by $G$ the complexification of $K$. Then for all $a\in \P$ the subgroup $G(a)\subset G$ is a Levi with set of simple roots $\{\alpha \in S\colon \<\alpha^{\vee},a\>=0\}$. After \cite[Theorem 11.2]{knop-autoHam}, we must prove that for all vertices $a$ of $\mathcal{P}$ the intersection $\Gamma = C_a\mathcal{P} \cap \wl_0$ is a smooth weight monoid for the reductive group $G(a)$. Notice that $\Z\Gamma=\wl_0$, since $\P$ is reflective and thus of maximal dimension, and that $\Z\Gamma$ is $W_a$-invariant by assumption.

Since $\P$ is reflective and Delzant with respect to $\wl_0$, the monoid $\Gamma$ is reflective and $E(\Gamma)$ is a basis of $(\Z\Gamma)^*$. The set of simple roots of $G(a)$ is contained in $\Z\Gamma$ by our assumptions. At this point all hypotheses of Theorem~\ref{thm:ourwoodward} are satisfied, therefore $\Gamma$ is smooth.
\end{proof}

We remark that the conditions on the lattice $\wl_0$ in the corollary are met if $S \inn \wl_0$ and $\wl_0$ is $W$-invariant. This corollary is slightly more general than Theorem~\ref{thm:woodward}, as it allows for lattices $\wl_0$ which are not equal to the weight lattice of $K$, see the examples below. 

\begin{example}
Let $G=\GL(2)$, fix $a\in \Z$ with $a \notin \{-1,0\}$, and define $\G=\la \w_1+a\w_2, \w_1-(a+1)\w_2\ra_\N$, where $\omega_i$ is the highest weight of $\bigwedge^i\C^2$ for all $i\in\{1,2\}$. We claim that $\G$ is smooth, and $\Z\G \neq \wl$. We show that $\Gamma$ is smooth by checking that it satisfies the assumptions of Theorem~\ref{thm:ourwoodward}. Set $\lambda_1 = \w_1+a\w_2$ and $\lambda_2=\w_1-(a+1)\w_2$. As $\a=\lambda_1+\lambda_2$, we have $\a\in \Z\G$, and since $s_\a(\lambda_1)=-\lambda_2, s_\a(\lambda_2)=-\lambda_1$, the lattice $\Z\Gamma$ is $W$-invariant.
The monoid $\G$ has full rank, and since the walls of $\G$ are generated by $\lambda_1$ and $\lambda_2$, one checks easily that $\G$ is reflective. As $\G$ is free, the set $E(\G)$ is a basis of $(\Z\G)^*$. So $\G$ fulfills all the assumptions of Theorem~\ref{thm:ourwoodward}, and hence $\G$ is smooth. Finally, it is elementary to check that $\Z\G$ does not contain $\omega_1$.
\end{example}

\begin{example}
We fix the lattice $\Z\Gamma$ with $\Gamma$ as in the above example. Consider the polytope $\P=\operatorname{conv}(0,\lambda_1, \lambda_2)$: it is elementary to check that it is Delzant in the vertices $\lambda_1$ and $\lambda_2$. Then, thanks to the above example and \cite[Theorem~11.2]{knop-autoHam}, the polytope $\P$ is the moment polytope of a Hamiltonian manifold for the group $\mathrm{U}(2)$ with nontrivial generic isotropy group.
\end{example}

\def\cprime{$'$} \def\cprime{$'$} \def\cprime{$'$} \def\cprime{$'$}
  \def\cprime{$'$}
\providecommand{\bysame}{\leavevmode\hbox to3em{\hrulefill}\thinspace}
\providecommand{\MR}{\relax\ifhmode\unskip\space\fi MR }
\providecommand{\MRhref}[2]{%
  \href{http://www.ams.org/mathscinet-getitem?mr=#1}{#2}
}
\providecommand{\href}[2]{#2}

\end{document}